\documentclass[11pt]{article}

\usepackage{amsmath}
\usepackage{epsfig, amssymb, amsfonts}
\usepackage{amsthm}
\usepackage{amstext}
\usepackage{amsopn}
\usepackage{enumerate}
\usepackage{mathrsfs}
\usepackage{subfigure}
\usepackage{graphicx}
\usepackage{color}
\usepackage{hyperref}
\usepackage[left=2.3cm,top=2cm,right=2.3cm,bottom=2cm, nohead,foot=1cm]{geometry}
\numberwithin{equation}{section}

\newtheorem{theorem}{Theorem}[section]
\newtheorem*{theorem*}{Theorem}
\newtheorem*{remark*}{Remark}
\newtheorem{lemma}[theorem]{Lemma}

\newtheorem{proposition}[theorem]{Proposition}
\newtheorem{definition}[theorem]{Definition}

\theoremstyle{definition}
\newtheorem{defn}[theorem]{Definition}

\theoremstyle{remark}
\newtheorem{remark}[theorem]{Remark}

\newcommand{\R}{{\mathbb R}}

\newcommand{\Z}{{\mathbb Z}}
\newcommand{\N}{{\mathbb N}}

\newcommand{\setb}[2]{\left \{ #1 \ \middle | \ #2 \right \} }

\DeclareMathOperator{\sgn}{sgn}

\makeatletter
\def\blfootnote{\gdef\@thefnmark{}\@footnotetext}
\makeatother

\title{Spectral analysis of a family of symmetric, scale-invariant diffusions with singular coefficients and associated limit theorems}

\author{\textbf{Jeremy T. Clark}\footnote{University of Mississippi, Department of Mathematics:  jeremy@olemiss.edu}  \quad   and \quad \textbf{Jeffrey H.  Schenker}\footnote{Michigan State University, Department of Mathematics:  jeffrey@math.msu.edu} \\ }

\begin{document}
\blfootnote{This work was supported by NSF Award DMS-0846325 and the Fund for Math}
\maketitle
\begin{abstract}
We discuss a family of time-reversible, scale-invariant diffusions with singular coefficients.  In analogy with the standard Gaussian theory, a corresponding family of generalized characteristic functions provides a useful tool for proving limit theorems resulting in non-Gaussian,  scale-invariant diffusions.  We apply the generalized characteristic functions in combination with a martingale construction to prove a simple invariance principle starting from a spatially inhomogeneous  nearest-neighbor random walk.  
\end{abstract}

\section{Introduction}
The determination of ``effective'' laws for suitable functionals of a collection of random variables is a basic problem in probability theory.  Generalized to stochastic processes, as in Donsker's invariance principle \cite{Donsker1951}, this problem becomes the question of whether a process follows an effective, and hopefully simpler, law when ``viewed from afar.''  The emergence of effective limiting laws over large scales is also a key problem in statistical physics and applied mathematics, in relation to coarse graining, renormalization group analysis, and universality of critical phenomena.    

A general perspective on this question is obtained by \emph{rescaling}. Given a process $t\mapsto \mathbf{X}_t$, where $t$ is a time parameter and $\mathbf{X}_t$ takes values in a vector space, one typically attempts to characterize large scale behavior by proving a limit theorem as $N\rightarrow \infty$ for the \emph{rescaled processes}  $N^{-\alpha} \mathbf{X}_{Nt}$. Here the \emph{characteristic exponent} $\alpha$ must be chosen precisely to give convergence in law to a non-trivial limit process $\mathbf{L}_t$.   A notable, if trivial, feature of this approach is that the law of the limit process will be \emph{scale-invariant}: $N^{-\alpha} \mathbf{L}_{Nt}$ must have the same distribution as $\mathbf{L}_t$. 

Thus processes with scale-invariant laws are singled out as the possible effective laws for the large-scale behavior of stochastic dynamics.  These scaling limits take a macroscopic view in which scale specific features of the dynamical law recede to a microscopic level as a scale-invariant picture emerges.  The most prominent example is Brownian motion (the Wiener process), which is the limit law for  Donsker's invariance principle \cite{Donsker1951}. Other well-known examples of scale-invariant limit laws include   Mittag-Leffler processes~\cite{Darling,Chen},  Brownian motions time-changed by Mittag-Leffler processes~\cite{Pap,Hopfner,ChenII},  Bessel processes~\cite{Csaki,Alexander,Lamperti}, and  stable processes~\cite{Mellet,Jara,Meershaert}.  For a classification of one-dimensional scale-invariant Markovian diffusions see \cite{Dante}.

The focus of this article is on scale-invariant diffusions $(\mathbf{x}_{t})_{t\geq 0}$ associated to Kolmogorov generators of the form
\begin{align}\label{Generator0}
L^{(\nu)} \ = \ \frac{1}{2}\frac{d}{dx}\frac{1}{|x|^{\nu }}\frac{d}{dx} \ , \qquad \text{for} \quad \nu>0.
\end{align}
These processes have the scale invariance
$$
\hspace{2.5cm}\mathcal{L}_{x}\big((\mathbf{x}_{t})_{t\geq 0}\big)\ = \ \mathcal{L}_{ N^{\frac{1}{\nu+2} } x}\left ( \left (N^{-\frac{1}{\nu+2}}\mathbf{x}_{Nt} \right )_{t\geq 0}\right )\ , \hspace{1cm} N>0   \ ,  
$$
where $\mathcal{L}_x$ denotes the law of the process with $\mathbf{x}_0=x$.  

For a given value of $\nu$, the process $(\mathbf{x}_{t})_{t\geq 0}$ formally satisfies the stochastic differential equation
$$d \mathbf{x}_t \ = \ \frac{1}{|\mathbf{x}_t|^{\frac{\nu}{2}}} d\omega_t \ - \ \nu \frac{\mathbf{x}_t}{|\mathbf{x}_t|^{\nu+2}} dt\ , $$
with $\omega_t$ a standard Brownian motion; although, the singular coefficients preclude directly integrating this equation to construct the process.  However, we may construct the process directly from the Dirichlet form associated to $L^{(\nu)}$; or, roughly equivalently,  from the transition kernel $\phi_t(x,x')$, which solves the heat equation
$$\partial_t \phi_t^{(\nu)}(x,x') \ = \ L^{(\nu)}\phi_t^{(\nu)}(x,x')\ , \quad \phi_0^{(\nu)}(x,x') = \delta(x-x')\ .$$
Remarkably, there is an explicit formula for $\phi_t^{(\nu)}$, see eq.\ \eqref{ExplicitSG}, involving modified Bessel function of the first kind of order $\pm \frac{\nu+1}{\nu+2}$.
%
A detailed construction and further properties of the processes $(\mathbf{x}_t)_{t\ge 0}$ may be found in Sec. \ref{SectGen} below. 

When $\nu=0$ the operator $L^{(\nu)}$ is the Laplacian, $\phi_t^{(\nu)}$ is the usual heat kernel, and the corresponding $(\mathbf{x}_t)_{t\ge 0}$ process is Brownian motion. In this case, a potent technique for deriving limit theorems is the \emph{method of characteristic functions}, which originated in Laplace's work on the central limit theorem. Given a process $(X_t)_{t\ge 0}$,  we define its characteristic function
\begin{equation}\label{characteristic}
 \varphi_t^{(0)}(q) \ = \ \mathbb{E} \left [ e^{i q X_t} \right ]  .
\end{equation}
An invariance principle, with Brownian motion as the limit, may be proved, in part, by showing that
$$\varphi_{Nt}^{(0)}(N^{-\frac{1}{2}} q) \ \xrightarrow[]{N\rightarrow \infty} \ e^{-Dt q^2}\ , $$
with $D$ a positive constant. (At a technical level, this only proves a central limit theorem for the random variable $N^{-\frac{1}{2}} X_{Nt}$; to obtain a limit law for the process we must consider characteristic functions depending on the values of the process at an arbitrary finite number of times. This extension is not difficult for Markov processes, such as considered below.)

The central point of the present paper is that a similar program is effective for limits converging to the process $(\mathbf{x}_t)_{t\ge 0}$ associated to $L^{(\nu)}$.  For each $\nu$, we define a generalized characteristic function using the eigenfunctions of $L^{(\nu)}$ in place of the complex-exponential eigenfunctions of $L^{(0)}$ in \eqref{characteristic}.  Specifically,
\begin{equation}
\label{gencharacteristic}
 \varphi_t^{(\nu)}(q) \ := \ \mathbb{E} \left [ \mathbf{e}^{(\nu)}( q X_t) \right ] ,
\end{equation}
where 
$$\mathbf{e}^{(\nu)} (x)\ :=\ C_{\nu}|x|^{ \frac{\nu+1}{2}}\left [ J_{-\frac{\nu+1}{\nu+2}}\Big(\frac{|x|^{\frac{\nu}{2}+1 }}{\frac{\nu}{2}+1}\Big)+\textup{i}\sgn(x)J_{\frac{\nu+1}{\nu+2} }\Big(\frac{|x|^{\frac{\nu}{2}+1 }}{\frac{\nu}{2}+1}\Big)    \right]  , $$
with $J_{\alpha}$ the ordinary Bessel function of the first kind of order $\alpha$.  As we will show below, for each $q$ the function $\mathbf{e}^{(\nu)}(qx)$ is an eigenfunction of $L^{(\nu)}$, with eigenvalue $-\frac{1}{2}|q|^{\nu+2}$, Furthermore, using known properties of the Hankel transform, it is straightforward to show that this is a complete set of eigenfunctions. 

\vskip .15in

\noindent \textbf{Limit Theorem}. To illustrate the utility of the generalized characteristic functions $\phi_t^{(\nu)}(q)$, we prove  Theorem~\ref{ThmWalkToDiff}, which  we summarize here as follows:
\begin{theorem*}
Let $(X_t)_{t\ge 0}$ be a continuous-time reversible random walk on $\Z$ with jump rates 
$$\mathbb{P}\left ( X_{t+dt}=n+1 \middle | X_t =n \right ) \ =  \ R_n dt\,,$$
where $$R_n \ = \ \frac{1}{2 |n|^\nu} \,+ \,\mathcal{O}\left ( \frac{1}{|n|^{\nu+1 }} \right )\ .$$
For each $N>0$, let $X^{(N)}_t = N^{-\frac{1}{\nu+2}} X_{Nt}$ where the process $X_t$ starts at $X_0=x_{N}$, with $x_{N} \in \Z$. If $N^{-\frac{1}{\nu+2}} x_{N} \rightarrow \mathbf{x}_0 \in \R$, then  
\begin{equation}\label{gencharlim} \mathbb{E} \left [ \mathbf{e}^{(\nu)} (  q X^{(N)}_t)\right ] \ \xrightarrow{N\rightarrow \infty} \ \mathbb{E} \left [ \mathbf{e}^{(\nu)}(q \mathbf{x}_t) \right ] \ = \ e^{-t \frac{|q|^{\nu+2}}{2}} \mathbf{e}^{(\nu)}(q \mathbf{x}_0) \ ,
\end{equation}
where the process $(\mathbf{x}_t)_{t\ge 0}$ has generator $L^{(\nu)}$ and starts from $\mathbf{x}_0$. It follows that
 $\big( X^{(N)}_t \big)_{t\ge 0}$ converges in law to $\big( \mathbf{x}_t \big)_{t\ge 0}$ as $N\rightarrow \infty$.
\end{theorem*}
\begin{remark*}
Using slightly messier estimates, we can replace $\mathcal{O}\big( |n|^{-\nu-1 } )$ in the asymptotics for $R_{n}$ above by $o\big (  |n|^{-\nu } \big)$.
\end{remark*}

We are particularly interested in the ``spectral'' technique described here because of the possibility that it may be generalized to settings for which stochastic tools are not available.   For instance, there are central limit theorems and invariance principles for the evolution of quantum observables, see, e.g., 
\cite{DRFP,KS}.  These were necessarily phrased in terms of  convergence of certain mean values, such as characteristic functions, since the precise value of an observable is given no meaning in quantum mechanics.  We believe the characteristic functions presented here will be useful in the analysis of the large scale behavior of certain physical systems.  We describe one such system below.
\vskip .15in
\noindent \textbf{Motivation.} 
Our interest in the processes $(\mathbf{x}_t)_{t\ge 0}$, stems from the apparent relation of the $\nu=2$ case to a functional limit theorem for the momentum of a particle moving  in a randomly shifting force field, in either classical or quantum mechanics.  To motivate what follows, let us start by sketching the connection. 

First consider the case of classical dynamics. Let $(Q_{t},P_{t})_{t\geq 0}$ be a stochastic process in $\R^{2}$ satisfying the pair of  differential equations 
\begin{align}\label{PerPot}
dQ_{t} \ = \ P_{t}dt    \qquad \qquad \text{and} \qquad \qquad  dP_{t}\ =\ -\frac{dV}{dq}\left (Q_{t}+\omega_{t}\right )dt \ ,
\end{align}
where $(\omega_{t})_{t\geq 0}$ is a standard Brownian motion and $V\in C^2(\R)$ is non-negative and periodic with period one.  The  equations~(\ref{PerPot}) describe a one-dimensional Newtonian particle with position $Q_{t}$ and momentum $P_{t}$ under the influence of a randomly shifting periodic force: $F_{t}(q):=-\frac{dV}{dq}(q+\omega_{t})$.  

The particle's energy is not conserved, as a consequence of the random shifting.  Instead Ito's formula implies that the energy  $H_{t}:=\frac{1}{2}P_{t}^{2}+V(Q_{t}+\omega_{t})$ obeys the differential equation 
\begin{align}\label{Hamil}
 dH_{t} \ = \ \frac{dV}{dq}(Q_{t}+\omega_{t})  d\omega_{t}  \ +\ \frac{1}{2}\frac{d^{2}V}{dq^{2}}(Q_{t}+\omega_{t})dt \ .    
\end{align}
Due to the continual input of randomness, the particle will typically accelerate to high energies $H_{t}\gg 1$ over long time periods.  At high energy the speed $|P_t|=\sqrt{2H_t -2V(Q_t + \omega_t)}$ is large and the argument $Q_{t}+\omega_{t}$ appearing in~(\ref{Hamil}) passes quickly through the period cells of $V(q)$ at an approximate frequency
$$\left (\int_{0}^{1}da \, \frac{1  }{ \sqrt{ 2H_{t}-2V(a) }} \right )^{-1}\ \approx \ \sqrt{ 2H_{t} } \ \gg \  1 \ .   $$
This swift cycling through the period cells generates an averaging effect  reminiscent of the adiabatic regimes studied by Freidlin and Wentzell~\cite{Wentzell} in which Hamiltonian flows are perturbed by  comparatively slow-acting white noises; it suggests that we may replace the coefficients $\frac{dV}{dq}(Q_{t}+\omega_{t}) $ and $\frac{1}{2}\frac{d^{2}V}{dq^{2}}(Q_{t}+\omega_{t})$ appearing  in~(\ref{Hamil}) with the respective high-energy averaged forms given by
\begin{align*}
   \sqrt{ 2H_{t} } \int_{0}^{1}da \, \frac{\big| \frac{dV}{dq}(a)   \big|^{2}  }{\sqrt{2H_{t}-2V(a) }  }& \ \approx \ \sigma  \quad  \text{and}\\      \sqrt{ 2H_{t} } \int_{0}^{1}da \, \frac{  \frac{1}{2}\frac{d^{2}V}{dq^{2} }(a)   }{ \sqrt{ 2H_{t}-2V(a) }  }  & \ = \       \sqrt{ 2H_{t} }\int_{0}^{1}da \, \frac{ -\frac{1}{2}\big| \frac{dV}{dq }(a) \big|^{2}  }{ \left ( 2H_{t}-2V(a) \right ) ^{\frac{3}{2}}  } \ \approx \ \frac{-\sigma}{4H_{t} } \ ,
\end{align*}
where $\sigma:=\int_{0}^{1}da\big|\frac{dV}{dx}(x)\big|^{2}$ and the equality follows from integrating by parts.  Thus, taking into account self-averaging at high energy, the equation~(\ref{Hamil}) is approximately equivalent to that of a dimension-$\frac{1}{2}$ Bessel process:
$
dH_{t}\approx  \sqrt{\sigma}d\omega'_{t} -  \frac{\sigma}{4H_{t}} dt   
$ where $\omega'$ is a standard Brownian motion.   

The above heuristic considerations suggest that the rescaled energy process  $(N^{-\frac{1}{2}}H_{\frac{Nt}{\sigma}})_{t\geq 0} $ approaches a  dimension-$\frac{1}{2}$ Bessel process in law as $N\rightarrow \infty$. The process $(\mathbf{x}_{t})_{t\ge 0}$, with generator $L^{(2)}$, has the property that $\frac{1}{2}\mathbf{x}_{t}^{2}$  is a dimension-$\frac{1}{2}$ Bessel process (see Proposition \ref{PropRelated} below).  Since the momentum $\mathbf{P}_t$ is  a signed process whose absolute value is approximately equal to $\sqrt{2H_{t}}$,  we are lead to conjecture  the following functional convergence for large $N$: 
\begin{align}\label{Ghosts}
\Big(  N^{-\frac{1}{4}}P_{\frac{Nt}{\sigma}}\ ,\  N^{-\frac{5}{4}}Q_{\frac{Nt}{\sigma}}    \Big)_{t\geq 0}\quad \stackrel{\frak{L}}{\Longrightarrow}  \quad \Big( \mathbf{x}_{t} \ , \ \int_{0}^{t}dr \, \mathbf{x}_{r}\Big)_{t\geq 0} \ .   
\end{align}
(Convergence of the rescaled position process would follow from that for the momentum process, since the position is merely a time integral of the momentum.)

We further expect some of what was said above to carry over to the quantum version of this  system. In that case, we cannot speak of the precise position $Q_t$ and $P_t$ of the particle at time $t$.  Instead, the state of the quantum system is described by a wave function $\psi(x,t)$ solving the time dependent Schr\"odinger equation
\begin{equation}
i \frac{\partial \psi(x,t)}{\partial t} \ = \  -\frac{1}{2} \frac{\partial^2 \psi(x,t)}{\partial x^2} + V(x +\omega_t) \psi(x,t) \ .
\end{equation}
The amplitude squared $|\psi(\cdot,t)|^2$ is interpreted, according to the axioms of quantum mechanics, as giving a probability density for the position $Q_t$ of the particle at time $t$.  The probability density for the momentum $P_t$ is given by $|\widehat{\psi}(\cdot,t)|^2$ where $\widehat{\psi}$ is the (ordinary) Fourier transform of $\psi(\cdot,t)$.  Although it no longer makes sense to talk about ``the momentum process,'' we can consider its generalized characteristic function
$$\varphi_t^{(2)}(q) \ := \ \int_{\R} \mathbf{e}^{(2)}(qp) |\widehat{\psi}(p,t)|^2 dp \ .$$
A semi-classical analysis of the system at high energies, leads us to conjecture that
\begin{equation}\label{QuantumGhosts}
\lim_{N\rightarrow \infty} \varphi_{\frac{Nt}{\sigma} }^{(2)}\big(N^{-\frac{1}{4}}q \big) \ = \ e^{-t \frac{|q|^4}{2}}\ .
\end{equation}
From this convergence would follow, for instance, the super ballistic propagation of the wave function:
$$\int_{\R} |x||\psi(x,t)|^2 dx \ \sim \ t^{\frac{5}{4}}\ .$$

The invariance principles~(\ref{Ghosts}) and \eqref{QuantumGhosts} will be the subject of forthcoming work. For the present paper, we direct our attention to the processes $(\mathbf{x})_{t\ge 0}$ generated by $L^{(\nu)}$, the associated generalized characteristic functions, and a limit theorem for a simple random walk with $ (\mathbf{x}_{t})_{t\geq 0}$ as a scaling limit law. 

The remainder of this article is outlined as follows:  In Sec.~\ref{SectGen} we discuss constructions and elementary properties of the process $(\mathbf{x}_{t})_{t\geq 0}$.   In Sec.~\ref{SectCharFun} we discuss the generalized characteristic function and the eigenfunctions of the generator $L^{(\nu)}$.  Sections~\ref{SectNNRW} contains a  proof of the  functional central limit theorem outlined above, yielding the law of  $(\mathbf{x}_{t})_{t\geq 0}$ starting from an inhomogeneous  simple random walk.

\section{Properties of the scale-invariant diffusion}\label{SectGen}

Since the generator~(\ref{Generator0}) is singular at zero,  we should be careful about how the corresponding process is defined.  There are, however, a number of constructions at our disposal.    In the following lemma we use results from~\cite{Fukushima} on stochastic processes determined by Dirichlet forms.   

\begin{lemma}\label{LemForm}
Let the Dirichlet form $\mathscr{E}:\mathscr{D}\times \mathscr{D}\rightarrow \R$ be defined  by $\mathscr{E}(u,v)=  \int_{\R}dx\frac{1}{|x|^{\nu}}\frac{du}{dx}(x)\frac{dv}{dx}(x) $ on the domain
\begin{align*}
  \mathscr{D} \ = \ \setb{ w\in L^{2}(\R)\cap H^{1}_{\textup{loc}}(\R)} {\int_{\R}dx\, \frac{1}{|x|^{\nu}}\left|\frac{dw}{dx}(x)\right|^{2} <\infty } .  
\end{align*}
The form $\mathscr{E}$ determines a strong Markov process $(\mathbf{x}_{t})_{t\in \R^{+}}$ with continuous trajectories.   The corresponding transition semigroup  is strongly continuous on $L^p(\R)$ for $1\le p<\infty$  and has explicit transition densities $\phi_{t}^{(\nu)}(x,x')$ of the form
\begin{align}\label{ExplicitSG}
 \phi_{t}^{(\nu)}(x,x') \ = \ \frac{|xx'|^{\frac{\nu+1}{2}} }{t(\nu+2)}e^{-\frac{|x|^{\nu+2}+|x'|^{\nu+2}      }{ 2t(\frac{\nu}{2}+1)^{2}  }}\left (  \mathcal{I}_{-\frac{\nu+1}{\nu+2}}\left( \frac{ |x x'|^{\frac{\nu}{2}+1}  }{ t(\frac{\nu}{2}+1)^{2}    } \right) \ + \ \sgn(xx') \, \mathcal{I}_{\frac{\nu+1}{\nu+2}}\left( \frac{ |x x'|^{\frac{\nu}{2}+1}     }{ t( \frac{\nu}{2}+1 )^{2}   } \right)   \right)   ,   
\end{align}
where $\sgn(y)=\frac{y}{|y|}$ and $\mathcal{I}_{\alpha}$ is the   modified Bessel function of the first kind of order $\alpha$. 
\end{lemma}
\begin{remark}
The explicit form of the transition semigroup~(\ref{ExplicitSG}) is similar to the explicit formula for Bessel processes~\cite[Appx. 7]{Revuz}.   Note that  for $x=0$ the expression~(\ref{ExplicitSG})  reduces to the form 
\begin{align}\label{Jaff}
 \phi_{t}^{(\nu)}(0,x') \ = \ N_{\nu}^{-1}t^{-\frac{1}{\nu+2}} e^{-\frac{ |x'|^{\nu+2 } }{ 2t (\frac{\nu}{2}+1)^{2}  }  } \hspace{1cm}\text{with}\hspace{1cm}  N_{\nu}\ := \ 2^{\frac{\nu+1}{\nu+2}}\, \Gamma\left(\frac{1}{\nu+2}\right)(\nu+2)^{-\frac{\nu}{\nu+2}}\ .  
\end{align}
Also note that  $ \phi_{t}^{(\nu)}(x,x')$ is nonnegative since $\mathcal{I}_{-\alpha }(r)\geq \mathcal{I}_{\alpha}(r)$ for all $r,\alpha>0$.  
\end{remark}

\begin{proof}
In the terminology of~\cite{Fukushima}, the form  $(\mathscr{E},\mathscr{D})$  is  regular and closed.    By~\cite[Thm 7.2.1]{Fukushima} there exists a symmetric Hunt process $(\mathbf{x}_{t})_{t\in \R^{+}}$ with corresponding form $\mathscr{E}$.   Since $\mathscr{E}(u,v)=0$ when $u,v\in   \mathscr{D}$ have disjoint compact supports, the trajectories of the process are continuous~\cite[Thm 4.5.1]{Fukushima}. Strong continuity of the corresponding semi-group on $L^p(\R)$  is a standard result for symmetric Markov semi-groups, see for example \cite[Thm. 1.4.1]{Davies}. 

The generator \eqref{Generator0} is defined by the standard Friedrichs construction on the domain
$$ \mathscr{D}(L^{(\nu)}) \ = \ \setb{ u\in \mathscr{D} }{ \mathscr{E}(v,u)\le C \|v\|_2 \text{ for some } C<\infty}\ .$$
The domain includes the set $\{ u\in H^1 \ | \ \frac{1}{|x|^\nu} \frac{du}{dx}\in H^1 \}.$
It follows for fixed $x\in \R$ and $t>0$, that $\phi_t^{(\nu)}(x,\cdot)\in \mathscr{D}(L^{(\nu)})$ and by explicit computation, using the modified Bessel equation
$ x^{2}\frac{d^{2}\mathcal{I}_{ \alpha } }{dx^{2} } +x\frac{d\mathcal{I}_{ \alpha } }{dx}=(x^{2}+\alpha^{2})\mathcal{I}_{ \alpha } $,  that  \begin{equation}\frac{d}{dt} \phi_{t}^{(\nu)}(x,\cdot) \ = \ L^{(\nu)} \phi_{t}^{(\nu)}(x,\cdot)\ .\label{eq:Heat}
\end{equation}
(There is an easier method to verify this differential equation, using the eigenfunctions of the generator $L^{(\nu)}$; see Sec.~\ref{SectCharFun}.)

For $\phi_{t}^{(\nu)}(x,x')$ to be the heat kernel, we need the initial condition $\lim_{t\searrow 0}\phi_{t}^{(\nu)}(x,\cdot)\stackrel{d}{=}\delta_{x}(\cdot)$  in addition to \eqref{eq:Heat}.  For $x=0$ the convergence can be shown using the form~(\ref{Jaff}).  For $x\neq 0$ the convergence can be shown using the asymptotic form of the modified Bessel functions  \cite[Eq. 9.7.1]{Handbook}:
$$\mathcal{I}_{\alpha }(r)\ =\ \frac{e^{r}}{ \sqrt{2\pi r}   } \left(1+\mathcal{O}(|r|^{-1})\right)\ , \qquad r\gg 1 .$$ It follows that  $\phi_{t}^{(\nu)}(x,x')$ may be approximated for small $t$ by
\begin{align}\label{NewExp}
\frac{|xx'|^{\frac{\nu}{4}}}{ 2(2\pi t)^{\frac{1}{2} }}   \big(1+\sgn(xx') \big)  e^{-\frac{\big(|x|^{\frac{\nu}{2}+1}-|x'|^{\frac{\nu}{2}+1} \big)^{2}    }{ 2t(\frac{\nu}{2}+1)^{2}  }} \  .
\end{align}
The distributional convergence of~(\ref{NewExp}) to $\delta_{x}(x')$ as $t\searrow 0$  is easily seen upon changing variables to $y'=\sgn(x')|x'|^{\frac{\nu}{2}+1}$.  
\end{proof}

The expression \eqref{ExplicitSG} for the semigroup  suggests an alternative construction of the process $\left ( \mathbf{x}_t \right )_{t\ge 0}$. Indeed,   \eqref{ExplicitSG} and the Markov property imply explicit formulae for all finite time marginals.  From these we obtain a version of the process by Kolomogorov extension.  It follows from well-known results, e.g., \cite[Theorem 3.26]{Liggett}, that there is a Feller process version of $(\mathbf{x}_t)_{t\ge 0}$.  The existence of a version with continuous sample paths follows from Kolomogorov's condition, see, e.g., \cite[Theorem 3.27]{Liggett}.

\begin{proposition}\label{PropRelated} Let $(\mathbf{x}_{t})_{t\geq 0}$ be the diffusion defined in Lem.~\ref{LemForm}.  Then
\begin{enumerate}[(1).]
\item   $\mathbf{m}_{t}:=\sgn(\mathbf{x}_{t}) |\mathbf{x}_{t}|^{\nu+1}$ is a martingale formally satisfying the stochastic differential equation
$$d\mathbf{m}_{t} \  = \ (\nu+1)\sgn( \mathbf{m}_{t})   |\mathbf{m}_{t}|^{\frac{\nu}{2\nu+2}} \, d\mathbf{\omega}_{t}\  . $$

\item $\mathbf{b}_{t} :=  \frac{2}{\nu+2}  |\mathbf{x}_{t}|^{\frac{\nu}{2}+1  }$ is a dimension-$\frac{2}{\nu+2}$ Bessel process:
$$d\mathbf{b}_{t} \  = \     d\omega_{t}   \ -\ \frac{\frac{\nu}{2\nu +4} }{\mathbf{b}_{t}}\,dt \ . $$

\item   $\mathbf{s}_{t}:=  \frac{4}{(\nu+2)^2} |\mathbf{x}_{t}|^{\nu+2}  $   is a dimension-$\frac{2}{\nu+2}$ squared Bessel process.  In particular, the increasing part in the Doob-Meyer decomposition for $\mathbf{s}_{t}  $ increases linearly: 
$$  d \mathbf{s}_{t}\ = \  2\sqrt{\mathbf{s}_{t}}\, d\omega_{t}+\frac{2}{\nu+2}\, dt \ .   $$
\end{enumerate}
In (1), (2) and (3), $\omega_t$ denotes a standard Brownian motion.
\end{proposition}

Parts (1) and (3) of Prop.~\ref{PropRelated} give key martingales related to $(\mathbf{x}_{t})_{t\geq 0}$.  In proving functional limit theorems that yield the law $(\mathbf{x}_{t})_{t\geq 0}$ as a limit in  Sec.~\ref{SectNNRW}, it will be useful to find analogous martingales defined in terms of the pre-limit processes.  The submartingale in part (3) may be used to understand the expected amount of time that $(\mathbf{x}_{t})_{t\geq 0}$ spends in regions around the origin.  For instance if $\varsigma_{a}$ is the time that $\mathbf{x}_{t}$ hits   $\pm a$ when starting from the origin, then 
$ \varsigma_a = \frac{\nu +2}{2} \left ( \mathbf{s}_{\varsigma_a} - M_{\varsigma_a} \right )$
where $M_t= 2 \int_0^t d\omega_r \, \sqrt{\mathbf{s}_r}$ is a martingale. The optional stopping theorem gives 
\begin{align}\label{UseSubMart}
\mathbb{E}_{0}[\varsigma_{a} ] \ = \ \frac{\nu+2}{2} \mathbb{E}_{0}[\mathbf{s}_{\varsigma_{a}} ] \ = \   \frac{2}{\nu+2}a^{\nu+2} \ .  
\end{align}
The Hausdorff dimension of the zero set for $(\textbf{x}_{t})_{t\geq 0}$ is a.s.\ $\frac{\nu+1}{\nu+2}$; see, e.g., \cite[p. 21]{Subord}. 

\section{Bessel characteristic functions}\label{SectCharFun}

For $\alpha\in \R$ let $J_\alpha:\R^{+}\rightarrow \R$ be the Bessel function of the first kind of order $\alpha$.  This is a solution of the Bessel equation 
\begin{align}\label{Bessel}
 x^{2}Z''(x)+ xZ'(x)+(x^{2}-\alpha^{2})Z(x) \ = \ 0     
 \end{align}
 with the asymptotic forms
 \begin{equation} J_{\alpha}(x) \ = \ \frac{1}{\Gamma(1+\alpha)}\left ( \frac{x}{2} \right )^{\alpha} + \mathcal{O}( x^{2+\alpha} ) \   ,\qquad \text{as $x\searrow 0$,}  \label{eq:0asymptotics} 
 \end{equation}
and 
 \begin{equation} 
 J_{\alpha}(x) \ = \ \sqrt{\frac{2}{\pi x}} \cos \left (x - \frac{\alpha \pi}{2}  - \frac{\pi }{4} \right ) + \mathcal{O} ( x^{-\frac{3}{2}} )\ , \qquad \text{ as $x \nearrow \infty$;} \label{eq:infityasymptotics} \end{equation}
see, e.g., \cite[Ch. 10]{Handbook}. For $\alpha \not \in \Z$, $J_{\alpha}$ and  $J_{-\alpha}$ are linearly independent solutions of \eqref{Bessel}.
 
For $\nu>0$ define $\mathbf{e}^{(\nu)} :\R\rightarrow \mathbb{C}$  as
\begin{align}
\mathbf{e}^{(\nu)} (x)\ :=\ u_{\nu}|x|^{ \frac{\nu+1}{2}}\left [ J_{-\frac{\nu+1}{\nu+2}}\bigg(\frac{|x|^{\frac{\nu}{2}+1 }}{\frac{\nu}{2}+1}\bigg)+\textup{i}\sgn(x)J_{\frac{\nu+1}{\nu+2} }\bigg(\frac{|x|^{\frac{\nu}{2}+1 }}{\frac{\nu}{2}+1}\bigg)    \right] ,
\end{align}
where the normalization constant  $u_{\nu}:=\Gamma(\frac{1}{\nu+2})(\nu+2)^{-\frac{\nu+1}{\nu+2}}$ is chosen so that $\mathbf{e}^{(\nu)} (0)=1$.  Note that $\mathbf{e}^{(\nu)}$ is $C^1$.  Indeed,  it is real analytic for $x\neq 0$ and continuously differentiable at $0$ due to the asymptotic eq.\ \eqref{eq:0asymptotics}. When $\nu=0$, $\mathbf{e}^{(0)}$ is a complex exponential $\mathbf{e}^{(0)} (x)= e^{\textup{i}x} $.  By eq.\ \eqref{eq:infityasymptotics},
\begin{align*}
\mathbf{e}^{(\nu)}(x) \ = & \ u_{\nu} \sqrt{\frac{\nu+2}{\pi}     }|x|^{\frac{\nu}{4}} \left [\cos\bigg(\frac{|x|^{\frac{\nu}{2}+1 }}{\frac{\nu}{2}+1}+\frac{\pi}{4}\frac{ \nu}{\nu+2}\bigg) \, +\, \textup{i}\sgn(x) \cos\bigg(\frac{|x|^{\frac{\nu}{2}+1 }}{\frac{\nu}{2}+1}-\frac{\pi}{4}\frac{3\nu+4}{\nu+2}\bigg)   \right ] \\ &\ + \ \mathcal{O} ( |x|^{-1 - \frac{\nu}{4}} )\ ,
\end{align*}
as $x\rightarrow \infty.$
Since $\mathbf{e}^{(\nu)}$ is continuous, there is $C_\nu<\infty$ such that
\begin{equation}
|\mathbf{e}^{(\nu)}(x)|\ \le \ C_\nu \big ( 1 + |x|^{\frac{\nu}{4}} \big )\,.
   \label{eq:easymptote}
\end{equation}

\begin{proposition} \label{PropEigenfunction}
Consider the one-parameter collection of functions $(\mathbf{e}_{q}^{(\nu)})_{q\in \R}$  defined by $\mathbf{e}^{(\nu)}_{q}(x):=\mathbf{e}^{(\nu)}(qx)$.   Then

\begin{enumerate}[(1).]
\item  $(\mathbf{e}_{q}^{(\nu)})_{q\in \R}$ is a complete set of  eigenfunctions for the generator $L^{(\nu)}$:
$$   L^{(\nu)}\mathbf{e}^{(\nu)}_{q}\ = \ -\frac{1}{2}|q|^{\nu+2}\mathbf{e}^{(\nu)}_{q}\ .    $$

\item  The functions $\mathbf{e}^{(\nu)}_{q}$ satisfy the orthogonality relation
$$  \int_{\R}dx  \, \overline{\mathbf{e}^{(\nu)}(qx)} \mathbf{e}^{(\nu)}(q'x) \  = \  4u_{\nu}^{2}\delta(q-q') $$
in the sense that for any $g\in L^2(\R)$, the limit
$$ \tilde{g}^{(\nu)}(q)  \ := \ L^2-\lim_{R\rightarrow \infty} \int_{|x|<R}  \overline{\mathbf{e}^{(\nu)}(qx)} g(x) $$
exists and the following generalized Plancherel formula holds for all $g,h\in L^2(\R)$,
$$ \int_\R dx\, \overline{g(x)}h(x)  \ = \ 4 u_{\nu}^2 \int_\R dq \, \overline{\tilde{g}^{(\nu)}(q)}  \tilde{h}^{(\nu)}(q)\ .$$

\item The heat kernel $\phi_{t}^{(\nu)}(x,x')$ can be written in the form
\begin{equation}  \phi_{t}^{(\nu)}(x,x')\ =\ \frac{1}{4u_{\nu}^{2}}\int_{\R}dq \,  e^{-\frac{t}{2}|q|^{\nu+2}} \mathbf{e}^{(\nu)}(qx)  \overline{\mathbf{e}^{(\nu)}(qx')} \  .   \label{eq:L2heatkernel}
\end{equation}


\end{enumerate}
\end{proposition}

\begin{proof} The functions $\mathbf{e}_{q}^{(\nu)}$ are seen to be eigenfunctions of $L^{(\nu)}$ by direct computation using the Bessel equation \eqref{Bessel}. The completeness and orthogonality relation  of  the collection $ (\mathbf{e}_{q}^{(\nu)})_{q\in \R}$ follow from the analogous $L^2$ theory for  Hankel transforms due to Watson \cite{WatsonL2}, see also \cite[Chap. VIII]{Titch},  since 
$$u_{\nu}|qx|^{ \frac{\nu+1}{2}} J_{-\frac{\nu+1}{\nu+2}}\Big(\frac{|qx|^{\frac{\nu}{2}+1 }}{\frac{\nu}{2}+1}\Big) \qquad \text{and}\qquad \textup{i}u_{\nu}\sgn(qx)|qx|^{ \frac{\nu+1}{2}}J_{\frac{\nu+1}{\nu+2} }\Big(\frac{|qx|^{\frac{\nu}{2}+1 }}{\frac{\nu}{2}+1}\Big)   $$
are the even and odd components of $\mathbf{e}^{(\nu)}_{q} (x)$, respectively, and the formal relation
\begin{align}\label{Sinbad}
 \int_{0}^{\infty}dx\, |qx|^{ \frac{\nu+1}{2}}  J_{\alpha }\Big(\frac{|qx|^{\frac{\nu}{2}+1 }}{\frac{\nu}{2}+1}\Big) |q'x|^{ \frac{\nu+1}{2}}J_{\alpha }\Big(\frac{|q'x|^{\frac{\nu}{2}+1 }}{\frac{\nu}{2}+1}\Big)  =   \delta\big(|q|-|q'|\big)
\end{align}
is equivalent through a change of variables to the usual form
$$   \int_{0}^{\infty}dx \, |q|^{\frac{1}{2}} J_{\alpha }\big(|qx|\big) |q'|^{\frac{1}{2}}J_{\alpha }\big(|q'x|\big) =\delta\big(|q|-|q'|\big)\ .  $$

The expression for the heat kernel follows from part (2).  The explicit formula  \eqref{ExplicitSG} for $\phi_t^{(\nu)}$ can be proved directly from \eqref{eq:L2heatkernel} using Weber's second exponential integral \cite[\url{http://dlmf.nist.gov/10.22.E67}]{DLMF}, see also \cite[\S 13.31]{Watson}.
\end{proof}

The definition of the generalized Fourier transform $g \mapsto \widetilde{g}^{(\nu)}$ in Prop. \ref{PropEigenfunction} makes use of the $L^2$ theory for generalized transforms of  \cite{WatsonL2}. It follows that the inverse transform is given by
\begin{equation} g(x) = L^2-\lim_{R\rightarrow \infty} \frac{1}{4 u_\nu^2}\int_{|q|<R} \mathbf{e}^{(\nu)}(qx) \widetilde{g}^{(\nu)}(q)\label{eq:inversetransform}
\end{equation}
for $g\in L^2(\R)$. In general the improper integral on the right hand side of \eqref{eq:inversetransform} and the improper integral defining $\widetilde{g}^{(\nu)}$ may not  converge absolutely.  Instead the $L^2$ limits exist due to the oscillations of $\overline{e}^{(\nu)}(x)$, despite the growth in magnitude of $|\overline{e}^{(\nu)}(x)|$ as $x\rightarrow \infty$ --- see  \eqref{eq:easymptote}. 

To proceed, we introduce the measure $dm(x)=(1+|x|^{\frac{\nu}{4}})dx$ and define \begin{equation} \widetilde{g}^{(\nu)}(q) \ = \ \int_{\R} dx \, \overline{\mathbf{e}^{(\nu)}(qx)} g(x)\ ,\label{eq:GenFour}
\end{equation}
 whenever $g\in L^1(dm)$. The integral on the right hand side converges absolutely by \eqref{eq:easymptote} and for $g\in L^2(\R) \cap L^1(dm)$ agrees with the previous definition of $\widetilde{g}^{(\nu)}$ by dominated convergence. As in the case of the usual Fourier transform, we may extend the transform to $g\in L^2(\R) + L^1(dm)$. 
 
\begin{lemma}\label{PropRiemannLebesgue}~
\begin{enumerate}[(1).]
\item There is a constant $C_\nu$ such that $|\widetilde{g}^{(\nu)}(q)| \le C_\nu \|g\|_{L^1(dm)} (1+|q|^{\frac{\nu}{4}})$
for all $g\in L^1(dm)$.
\item
\emph{(Riemann-Lebesgue Lemma)}  If $ g\in L^1(dm)$ then $\widetilde{g}^{(\nu)}$ is continuous and
\begin{equation}\lim_{q\rightarrow \infty} \frac{\widetilde{g}^{(\nu)}(q)}{1 + |q|^{\frac{\nu}{4}}} \ = \ 0 \ .\label{eq:RL}
\end{equation}
\item If $g\in \mathcal{D}(L^{(\nu)})\cap L^1(dm)$ and $L^{(\nu)} g \in L^1(dm)$, then there is $C<\infty$ such that $$|\widetilde{g}^{(\nu)}(q)|\le C \frac{1}{1 + |q|^{\frac{3\nu}{4} +2}}\  ,$$
in particular $\widetilde{g}^{(\nu)} \in L^1(dm)$.
\item If $g\in L^2(\R) + L^1(dm)$ and  $\widetilde{g}^{(\nu)}\in L^1(dm)$, then for a.e.\ $x\in \R$
$$g(x)= \frac{1}{4u_\nu^2} \int_{\R}dq\,  \mathbf{e}^{(\nu)}(qx) \widetilde{g}^{(\nu)}(q)\  .$$
\end{enumerate}
\end{lemma}
\begin{remark}
Note that the hypothesis of part (3) holds for any $g \in C^2_c(\R)$ such that $g$ is constant in a neighborhood of $0$.\end{remark}
\begin{proof}
Part (1) is a consequence of \eqref{eq:easymptote}. 
 
Turning to part (2), note that the continuity of $\widetilde{g}^{(\nu)}$ for $g\in L^1(dm)$ follows from the continuity of $\mathbf{e}^{(\nu)}(qx)$ and dominated convergence.  The estimate \eqref{eq:RL} follows from part (1) and part (3) in much the same way as the usual Riemann-Lebesgue lemma is proved.  First note that \eqref{eq:RL} is trivial for a function in $C^2_c(\R\setminus \{0\})$ on account of the remark following the lemma and part (3). Since  $g\in L^1(dm)$ may be approximated in the $L^1(dm)$ norm as well as we like using such functions, we conclude from part (1) that $\limsup_q (1+|q|)^{-\frac{\nu}{4}} |\widetilde{g}^{(\nu)}(q)| <\epsilon$ for any $\epsilon$.  

To prove the bound in part (3), first note that for $g\in \mathcal{D}(L^{(\nu)})$ we have
$$\widetilde{\left [ L^{(\nu)} g\right]}^{(\nu)}(q) = -\frac{1}{2} |q|^{\nu +2} \widetilde{g}^{(\nu)}(q)\ .$$
If furthermore $L^{(\nu)}g,\ g \in L^1(dm)$ we conclude that $\left ( 1 + |q|^{(\nu+2)} \right ) |\widetilde{g}^{(\nu)}(q)| \le C (1+ |q|^{\frac{\nu}{4}})$, and the estimate follows.

Finally, let $g\in L^2(\R)+L^1(dm)$ and suppose $\widetilde{g}^{(\nu)}\in L^1(dm)$. Then $g\in L^2(\R)+L^1(\R)$. Since $L^{(\nu)}$ generates a strongly continuous contractive semigroup on $L^1$ as well as $L^2$, the absolutely convergent integral 
$$g_t(x)\ =\ \int_{\R}dx'\, \phi_t(x,x') g(x') $$
defines a family of functions that converge to $g$ in the $L^2(\R)+L^1(\R)$ norm as $t \rightarrow 0$.\footnote{Recall that the $L^2(\R)+L^1(\R)$ norm of $f$ is the infimum over $\|\phi\|_2 + \|\psi\|_1$ where $f=\phi + \psi$.} However for each $t>0$ we have 
$$g_t(x) \ = \  \int_{\R} dq\, \mathbf{e}^{(\nu)}(xq) e^{-\frac{t}{2}|q|^{\nu+2} } \widetilde{g}^{(\nu)}(q)$$
by part (3) of Prop.\ \ref{PropEigenfunction}. Passing to a subsequence $t_j\rightarrow 0$, we obtain a sequence such that $g_{t_j}(x)\rightarrow g(x)$ almost everywhere and the result follows by dominated convergence. 
\end{proof}

In the proofs of our limit theorems, we use the following generalized characteristic function of a probability measure, defined in terms of the eigenfunctions $\mathbf{e}^{(\nu)}_q$.
\begin{definition}[$L^{(\nu)}$-characteristic functions]
Let the measure $\mu\in \mathcal{M}_{1}(\R)$ satisfy $\int_{\R}d\mu(x)|x|^{\frac{\nu}{4}}<\infty$.  We define the $L^{(\nu)}$-characteristic function  by
$$\varphi_{\mu}^{(\nu)}(q):=\int_{\R}d\mu(x) \mathbf{e}^{(\nu)}(qx) \ .      $$
\end{definition}
By the remark above Prop.~\ref{PropEigenfunction}, $\varphi_{\mu}^{(\nu)}$ is equal to the standard characteristic function for $\nu=0$.   Our main use of the $L^{(\nu)}$-characteristic function is as a tool to prove vague convergence of measures.
\begin{proposition}\label{PropBesselConv} Let $\{\mu_{n}\}_{n\in \mathbb{N}\cup \{ \infty\}  }$  be probability measures on $\R$ satisfying $\sup_{n}\int_{\R}d\mu_{n}(x)|x|^{\frac{\nu}{4}}<\infty$.   If $\varphi_{\mu_n}^{(\nu)}$ converges pointwise to $\varphi_{\mu_{\infty}}^{(\nu)}$, then $\mu_{n}$ converges vaguely to $\mu_{\infty}$ as $n\rightarrow \infty$.  

\end{proposition}
\begin{proof} It suffices to show that $\int_\R d\mu_{n}(x)g(x)$ converges to $\int_\R  d\mu_{\infty}(x)g(x)$ for all $g$ in a subset $\mathcal{S}\subset C_0(\R)$ dense in $C_0(\R)$ in the uniform norm. A convenient choice of $\mathcal{S}$ is the set mentioned in the remark following Lem.~\ref{PropRiemannLebesgue}, i.e.,
$$ \mathcal{S} = \setb{g\in C_c^{2}(\R)}{ \text{for some $\epsilon >0$ } g(x) =g(0) \text{ if $|x|<\epsilon$}}.$$
Then $\mathcal{S}$ is easily seen to be dense in $C_0(\R)$ and, furthermore  $\mathcal{S} \subset \mathcal{D}(L^{(\nu)})\cap L^1(dm)$ and $L^{(\nu)} \mathcal{S} \subset L^1(dm)$.  Thus by part (3) of Lem.~\ref{PropRiemannLebesgue} we conclude that $\widetilde{g}^{(\nu)}\in L^1(dm)$ whenever $g\in \mathcal{S}$.  Hence, by part (4) of Lem.~\ref{PropRiemannLebesgue} and Fubini's Theorem,
$$\int_\R d\mu_n(x)   g(x) \ = \ \int_\R d\mu_n(x) \int_{\R} dq  \left [ \mathbf{e}^{(\nu)}(q x) \widetilde{g}^{(\nu)}(q)\right ]  \ = \ \frac{1}{4 u_\nu^2} \int_\R dq \ \widetilde{g}^{(\nu)}(q) \phi_{\mu_n}^{(\nu)}(q)\ ,\quad \text{ if } g\in \mathcal{S}\ .$$
Since $|\phi_{\mu_n}^{(\nu)}(q)| \le (1+|q|^{\frac{\nu}{4}})\sup_n \int_\R d\mu_n(x) (1 + |x|^{\frac{\nu}{4}})$ the result follows by dominated convergence.
\end{proof}

 We also define $\mathbf{f}_{q}^{(\nu)}(y):=\mathbf{e}_{q}^{(\nu)}\big(s(y)|y|^{\frac{1}{\nu+1}}\big)$.  These  are eigenfunctions for the backwards generator $G^{(\nu)}:= \frac{1}{2} (\nu+1)^{2}|y|^{\frac{\nu}{\nu+1}}\frac{d^{2}}{dy^{2}}$ of the martingale $\mathbf{m}_{t}$ defined in Prop.~\ref{PropRelated}:
\begin{align}\label{Habble} -|q|^{\nu+2}\mathbf{f}_{q}^{(\nu)}(y)\ = \ (\nu+1)^{2}|y|^{\frac{\nu}{\nu+1}} \frac{d^{2}\mathbf{f}_{q}^{(\nu)}}{du^{2} }(y) \  . 
\end{align}
Parts (2) and (3) of the proposition below list some useful asymptotic bounds for the derivatives of the functions  $\mathbf{e}_{q}^{(\nu)}$ and     $\mathbf{f}_{q}^{(\nu)}$.

\begin{proposition}\label{PropEigenKets}\text{ }

\begin{enumerate}[(1).]

\item  For any $n\in \N$, $\nu\in \R^{+}$ and $q\in \R$, there is a $C_{n,\nu,q}>0$ such that for all $x\in \R$ 
$$ \Big| \frac{d^{n}}{dx^{n}}\mathbf{e}_{q}^{(\nu)}(x)\Big|  \ \leq \ C_{n,\nu,q}\left(1+ |x|^{\nu+1-n}1_{|x|\leq 1} + |x|^{\frac{\nu}{2}n+\frac{\nu  }{4   }}1_{|x|\geq 1}\right) . $$

\item  For any $n\in \N$, $\nu\in \R^{+}$ and $q\in \R$, there is a $C_{n,\nu,q}>0$ such that for all $y\in \R$
$$ \Big|\frac{d^{n}}{dy^{n}}\mathbf{f}_{q}^{(\nu)}(y) \Big| \ \leq \ C_{n,\nu,q} \left( (\delta_{n,0}+  |y|^{\frac{\nu+2}{\nu+1}-n} )  1_{|y|\leq 1}   + |y|^{\frac{\nu}{2\nu+2}(\frac{1}{2}-n) } 1_{|y|\geq 1}    \right). $$

\end{enumerate}

\end{proposition}

\begin{proof}  This follows from the derivative formula \cite[\url{http://dlmf.nist.gov/10.6.E2}\ ]{DLMF} and the asymptotic forms \eqref{eq:0asymptotics} and \eqref{eq:infityasymptotics}.
\end{proof}

%


\section{An invariance principle for a nearest-neighbor random walk}\label{SectNNRW}

In this section we will consider the long-time limiting behavior for a continuous-time random walk $(X_{t})_{t\geq 0}$ on $\Z$
with generator $L_{R}$ operating on functions $F:\Z\rightarrow \R$ as
\begin{align}\label{Oreilly}
(L_{R}F)(n)\ = \ R_{n}^{-}\big[F(n-1)-F(n)    \big]+R_{n}^{+}\big[  F(n+1)-F(n)    \big] \ .
\end{align}
We suppose that the  jump rates $R_{n}^{\pm}>0$  satisfy the symmetry $R_{n}^{+} =R_{n+1}^{-}$ and  have the asymptotic form
\begin{align}\label{JumpAsymptotics}
 R_{n}^{+}\ = \ \frac{1}{2|n|^{\nu}}\,+\,\mathcal{O}\Big( \frac{1}{|n|^{\nu+1}}  \Big)\ , \qquad |n|\gg 1. 
\end{align}
The process $(X_{t})_{t\geq 0}$ is a simple,  time-reversible  random walk for which the invariant measure is counting measure.  The following limit theorem is the main result of this section.  
\begin{theorem}\label{ThmWalkToDiff} Define $X_{t}^{(N)}:=N^{-\frac{1}{\nu+2}}X_{Nt}$  where $X_{t}$ is a random walk generated by~(\ref{Oreilly}) with initial condition $X_{0}=x_N$.  If   $N^{-\frac{1}{\nu+2}}x_N\rightarrow \widehat{\mathbf{x}}$ 
 as $N\rightarrow \infty$, then there is  convergence in law as processes over any bounded time interval:
   $$X_{t}^{(N)} \,\,\, \stackrel{\frak{L}}{\Longrightarrow}  \,\,\, \mathbf{x}_{t} \,,$$
  where  $\mathbf{x}_0=\widehat{\mathbf{x}}$.  The convergence in law above is with respect to the uniform metric on paths. 
\end{theorem}

The convergence in law for the processes will follow by standard techniques once we verify the  one-dimensional convergence in law of $X^{(N)}_{t}$ to $\mathbf{x}_{t}$ as $N\rightarrow \infty$ for a single time $t$.   By Prop.~\ref{PropBesselConv},  $X^{(N)}_{t}$ converges in law to  $\mathbf{x}_{t}$  if there is convergence as $N\rightarrow \infty$ of the generalized characteristic functions   for all  $ q\in \R$:
\begin{align}\label{eq:XtNtoxt}
\mathbb{E}_{ \widehat{x}_{N} }\big[\mathbf{e}_{q}^{(\nu)}\big(X^{(N)}_{t}\big)\big]\ \longrightarrow  \ \mathbb{E}_{\widehat{\mathbf{x}} }\big[\mathbf{e}_{q}^{(\nu)}(\mathbf{x}_{t})\big]\quad \text{for}\quad  \widehat{x}_{N}\ :=\ N^{-\frac{1}{\nu+2}}x_N \ . 
\end{align}
Here and below the subscript $a\in \R$ of an expectation  $\mathbb{E}_{a}$  refers to the initial value of  whichever Markovian process happens to sit in the argument of the expectation.
  
 Define $Y(x):= \sgn(x)|x|^{\nu +1}$ and    $\widehat{\mathbf{y}}:=Y(\widehat{\mathbf{x}})   $.   Note that $\mathbb{E}_{ \widehat{\mathbf{x}} }\big[\mathbf{e}_{q}^{(\nu)}(\mathbf{x}_{t})\big] \ = \ \mathbb{E}_{\widehat{\mathbf{y}}} \big [ \mathbf{f}_q^{(\nu)}(\mathbf{m}_t)\big]$, where the functions $\mathbf{f}_q^{(\nu)}$ were defined in Sec.\ \ref{SectCharFun} and $\mathbf{m}_t= Y(\mathbf{x}_{t})$ is the martingale defined in Prop.\ \ref{PropRelated}. Similarly, 
$ \mathbb{E}_{\hat{x} }\big[\mathbf{e}_{q}^{(\nu)}\big(X^{(N)}_{t}\big)\big] \ = \ \mathbb{E}_{\hat{x}}\Big[\mathbf{f}_{q}^{(\nu)}\Big(Y\big(X^{(N)}_{t}\big)\Big)\Big]$.  Thus we may establish \eqref{eq:XtNtoxt} by bounding the difference
$$ \left | \mathbb{E}_{\widehat{x}_{N}}\Big[\mathbf{f}_{q}^{(\nu)}\Big(Y\big(X^{(N)}_{t}\big)\Big)\Big]\ - \  \mathbb{E}_{\widehat{\mathbf{y}} }\big [ \mathbf{f}_q^{(\nu)}(\mathbf{m}_t)\big] \right | \ , $$
where   to accomplish this we will further approximate $Y\big(X^{(N)}_t\big)$ with a martingale $M^{(N)}_t$ defined below.  To this end we make the following observations:

\begin{proposition} \label{PropBasicoDos}  \text{  }

\begin{enumerate}[(1).]
\item   The process $M_{t}:=\widehat{Y}(X_{t})$ is a martingale for $\widehat{Y}:\Z\rightarrow \R$ defined by $$\displaystyle \widehat{Y}(n):= \sigma (\nu+1) \sum_{m=1}^{|n|}\frac{1}{R_{\sigma m}^{-\sigma} }\,,$$ where $\sigma = \pm $ is the sign of $n\in \Z$.  

\item  The predictable quadratic variation for $ M_{t}$ has the form $\langle M\rangle_{t}=\int_{0}^{t}dr\widehat{Q}(X_{r})$ for $\widehat{Q}:\Z\rightarrow \R^+$ defined by $$\widehat{Q}(n)\,:=\, (\nu+1)^2\Big( \frac{1}{ R_{n}^{+}}\,+\, \frac{1}{R_{n}^{-}} \Big)\,.$$ 

\item The predictable, increasing part in the Doob-Meyer decomposition of the submartingale  $ |M_{t}|^{\frac{\nu+2  }{ \nu+1 }}$ has the form $\int_{0}^{t}dr\widehat{A}(X_{r})$ for 
$$  \widehat{A}(n)\ := \ R_{n}^{-}\Big[\big|\widehat{Y}(n-1)\big|^{\frac{\nu+2  }{ \nu+1 }}-\big|\widehat{Y}(n)\big|^{\frac{\nu+2  }{ \nu+1 }}   \Big]\ +\ R_{n}^{+}\Big[  \big|\widehat{Y}(n+1)\big|^{\frac{\nu+2  }{ \nu+1 }}-\big|\widehat{Y}(n)\big|^{\frac{\nu+2  }{ \nu+1 }}   \Big] \,  .$$

\end{enumerate}

\end{proposition}
\noindent The proof is essentially by direct computation.  We see that $M_t$ is a martingale because $L_R \widehat{Y}=0$.

\subsection{Proof of Theorem~\ref{ThmWalkToDiff}} \label{SecDiffProof}
We begin by stating certain estimates required in the proof.  The proofs of these technical lemmas are in Sec.\ \ref{SecDiffTech}. Throughout we will use $C$ to denote an arbitrary positive finite constant that may depend on the order $\nu$ and the time interval $[0,T]$ but which is independent of other parameters, unless otherwise indicated. The value of $C$ may change from line to line.

  The following lemma gives bounds on the difference between the functions  $ Y_{N}(x)$, $Q_{N}(x)$ and their respective  $N\rightarrow \infty$ limits.  Let $\widehat{Y}(\Z)$ denote the image of $\widehat{Y}$ and let  $\widehat{W}:\widehat{Y}(\Z)   \rightarrow   \Z  $ be the function inverse of $ \widehat{Y} $.   Define $Q:\R\rightarrow [0,\infty)$ by $Q(y):=(\nu+1)^2|y|^{\frac{\nu}{\nu+1}}$, and notice that $Q(y)$ is the diffusion coefficient for $(\frak{m}_{t})_{t\geq 0}$ in~(\ref{Habble}).

\begin{lemma}\label{SecLemMisc}
Let $Y_{N}: N^{-\frac{1}{\nu+2}} \Z  \rightarrow \R $ and  $Q_{N},\,A_{N}: N^{-\frac{\nu+1}{\nu+2}}\widehat{Y}(\Z) \rightarrow \R    $  be defined through
$$Y_{N}(x)\ :=\ N^{-\frac{\nu+1}{\nu+2}} \widehat{Y}\big(  N^{\frac{1}{\nu+2}} x  \big),  \hspace{.5cm}  Q_{N}(y)\ :=\  N^{-\frac{\nu}{\nu+2}}\widehat{Q}\Big( \widehat{W}\big(  N^{\frac{\nu+1}{\nu+2}}y  \big)    \Big),\hspace{.5cm}A_{N}(y)\ := \ \widehat{A}\Big( \widehat{W}\big(  N^{\frac{\nu+1}{\nu+2}}y  \big)    \Big)\    .$$
Then there exists  $C>0$ such that  the following inequalities hold for all  $N>1$, $x\in N^{-\frac{1}{\nu+2}} \Z$, and $y\in N^{-\frac{\nu+1}{\nu+2}}\widehat{Y}(\Z)$:
\begin{enumerate}[(1).]

\item $\big|  Y_{N}(x) -Y(x)  \big|\ \leq \  CN^{-\frac{1}{\nu+2}}  |x|^{\nu}    $

\item $\big|   Q_{N}(y)-  Q(y)   \big|\ \leq \   C N^{-\frac{\nu}{\nu+2}}\delta_0(y)  \ +\ CN^{-\frac{1}{\nu+2}} |y|^{\frac{\nu-1}{\nu+1} }\big ( 1-\delta_0(y))   $

 \item   $A_{N}(y)\ > \ C^{-1}$ 

\item    $|  Y(x) |\ \leq \  C|Y_{N}(x)| $,\quad $ |  Y_{N}(x) |\ \leq \ C|Y(x)|$,   \quad   $Q_{N} (y)\ \leq \  C\big(1+Q(y)\big)$

\end{enumerate}
\end{lemma}

To state the next lemma, we make the following definition. 
\begin{defn}
Suppose given for each $N>1$ a real-valued stochastic process $\big( X_t^{(N)};t\in[0,T]\big)$.  We say that the family $\big\{ \big (X_t^{(N)} ;t\in [0,T] \big ) \,\big| \, N>1 \big\}$ is \emph{$C$-tight at infinity} if for any sequence $N_j \rightarrow \infty$,
\begin{enumerate}[(1).]
\item  the sequence $\big ( X_t^{(N_j)};t\in [0,T] \big )$ is tight with respect to the uniform metric, and
\item if  $\big ( X_t^{(N_j)};t\in [0,T] \big )$ converges in law to some limit, then the limit process has continuous sample paths.
\end{enumerate}
\end{defn}

The next lemma states uniform in $N$ bounds on the moments of $X_r^{(N)}$ and $M_r^{(N)}$ and the associated tightness of the processes.

\begin{lemma}\label{SecMomentLem} Let $M_{t}^{(N)}:= Y_{N}\big(X^{(N)}_{t}\big)$.
\begin{enumerate}[(1).]
\item   For any $n\geq 1$ there is a $C_{n}>0$ such that for all  $t\in \R^{+}$, $x\in \R$, and $N>1$
$$   \mathbb{E}_{x}\bigg[ \sup_{0\leq r\leq t} \big|X_{r}^{(N)}\big|^{n}   \bigg]\ \leq \ C_{n}\big(1\,+\,|x|^{n}\ +\ t^{\frac{n}{\nu+2}}  \big) \  .$$

\item  There  is a $C>0$ such that  for all  $t\in \R^{+}$, $y\in \R$, $0<\alpha\leq 1$, and  $N>1$
$$  \int_{0}^{t}dr   \mathbb{P}_{y}\big[|M_{r}^{(N)}|  \leq   N^{-\alpha}    \big] \  \leq \ C N^{\frac{-\alpha}{\nu+1}} \big( 1+t^{\frac{\nu+1}{\nu+2}}   \big)\ . $$

\item   Suppose that the initial values $\widehat{x}_{N}:=X_{0}^{(N)}$ are uniformly bounded  for all  $N>1$.   The family of processes $\big( X_{t}^{(N)};  t\in  [0,T]\big) $ for $N>1$ is $C$-tight at infinity.

\end{enumerate}
\end{lemma}

\vspace{.3cm}

Now we proceed with the proof of Thm.~\ref{ThmWalkToDiff}.  To begin we prove convergence in law at a single time,  \eqref{eq:XtNtoxt}.  Denote $ \widehat{x}_{N}:= N^{-\frac{1}{\nu+2}}x_{N} $.   The difference between $\mathbb{E}_{\widehat{x}_{N}}\big[\mathbf{e}_{q}^{(\nu)}\big(X^{(N)}_{t}\big)\big]$ and $\mathbb{E}_{\widehat{\mathbf{x}} }\big[\mathbf{e}_{q}^{(\nu)}(\mathbf{x}_{t})\big]$ is bounded by the sum of terms
\begin{align}\label{Zing}
\Big| \mathbb{E}_{\widehat{x}_{N}}\big[\mathbf{e}_{q}^{(\nu)}\big(X^{(N)}_{t}\big)\big]\ - \ \mathbb{E}_{\widehat{\mathbf{x}} }\big[\mathbf{e}_{q}^{(\nu)}(\mathbf{x}_{t})\big] \Big|\leq \  & \bigg| \mathbb{E}_{\widehat{x}_{N}}\Big[\mathbf{f}_{q}^{(\nu)}\Big(Y\big(X^{(N)}_{t}\big)\Big)\Big]- \ \mathbb{E}_{\widehat{x}_{N}}\Big[\mathbf{f}_{q}^{(\nu)}\Big( Y_{N}\big(X^{(N)}_{t}\big)\Big)\Big]      \bigg|\nonumber  \\ &+ \Big| \mathbb{E}_{\widehat{y}_{N}}\big[\mathbf{f}_{q}^{(\nu)}\big(M^{(N)}_{t}\big)\big]  \ -  \ \mathbb{E}_{\widehat{y}_{N}}\big[\mathbf{f}_{q}^{(\nu)}(\mathbf{m}_{t})\big]     \Big|\nonumber \\ & +\Big|  \mathbb{E}_{\widehat{y}_{N}}\big[\mathbf{f}_{q}^{(\nu)}(\mathbf{m}_{t})\big] \   -  \  \mathbb{E}_{\widehat{\mathbf{y}}}\big[\mathbf{f}_{q}^{(\nu)}(\mathbf{m}_{t})\big] \Big|  \  ,
\end{align}
where $Y(x):=\sgn(x)|x|^{\nu+1}$, $\widehat{y}_{N}:= Y_{N}(\widehat{x}_{N})  $, and $ \widehat{\mathbf{y}} := Y(\widehat{\mathbf{x}})  $.   In~(\ref{Zing}) we have used that $M^{(N)}_{t} =  Y_{N}\big(X^{(N)}_{t}\big)$  and $\mathbf{f}_{q}^{(\nu)} \big(Y(x)\big)=\mathbf{e}_{q}^{(\nu)}(x)    $.   We will bound the three terms on the right side of~(\ref{Zing}) in (\textup{i})-(\textup{iii}) below.  \vspace{.2cm}

\vspace{.3cm}


\noindent (\textup{i}).  The first term on the right side of~(\ref{Zing}) is smaller than 
\begin{align}
\bigg| \mathbb{E}_{\widehat{x}_{N}}\Big[\mathbf{f}_{q}^{(\nu)}\Big(Y\big(X^{(N)}_{t}\big)\Big)\Big]\ -\  \mathbb{E}_{\widehat{x}_{N}}\Big[\mathbf{f}_{q}^{(\nu)}\Big( Y_{N}\big(X^{(N)}_{t}\big)\Big)\Big]      \bigg|\leq \ & \bigg(\sup_{y\in \R}\bigg| \frac{d\mathbf{f}_{q}^{(\nu)}}{dy}(y)\bigg|\bigg)\mathbb{E}_{\widehat{x}_{N}}\left[    \Big|Y\big(X^{(N)}_{t}\big) -Y_{N}\big(X^{(N)}_{t}\big)\Big|\right]\nonumber \\ \ \leq \ &  CN^{-\frac{1}{\nu+2}} \bigg(\sup_{y\in \R}\bigg| \frac{d\mathbf{f}_{q}^{(\nu)}}{dy}(y)\bigg|\bigg)\mathbb{E}_{\widehat{x}_{N}}\Big[    \big|X^{(N)}_{t}\big|^{\nu} \Big]\nonumber \\ \ = \ & \mathcal{O}\big(  N^{-\frac{1}{\nu+2}}  \big) \  .
\end{align}
The derivative of $\mathbf{f}_{q}^{(\nu)}$ is uniformly bounded by part (2) of Prop.~\ref{PropEigenKets}, the second inequality follows from part (1) of Lem.~\ref{SecLemMisc}, and the expectation on the second line is uniformly bounded as $N\rightarrow \infty$ by part (1) of Lem.~\ref{SecMomentLem}.\vspace{.3cm}

\noindent Part (ii):  Since  the function $ Y_{N}$ is one-to-one, $M_{t}^{(N)}$ is a Markov process, and we denote its backwards generator by $G^{(N)}_{R}$. So $G^{(N)}_R$ is a linear operator that acts as follows
\begin{equation}\label{eq:welcometothejungle}
\big(G^{(N)}_{R} F \big)\bigg( \frac{ \widehat{Y}(n)}{  N^{\frac{\nu+1}{\nu+2}} }\bigg) \ =  \   \sum_{\pm} NR_{n}^{\pm}\bigg[ F\bigg( \frac{ \widehat{Y}(n\pm 1)}{  N^{\frac{\nu+1}{\nu+2}} }\bigg)- F\bigg( \frac{ \widehat{Y}(n)}{  N^{\frac{\nu+1}{\nu+2}} }\bigg)   \bigg]
\end{equation}
on functions $F$ defined on the state space of $M_{t}^{(N)}$, which is the set $N^{-\frac{\nu+1}{\nu+2}} \widehat{Y}(\Z)$.  We can expand the difference of semi-groups with a Duhamel formula to get
\begin{align}
 \mathbb{E}_{\widehat{y}_{N}}\big[\mathbf{f}_{q}^{(\nu)}\big(M^{(N)}_{t}\big)\big]   -   \mathbb{E}_{\widehat{y}_{N} }\big[\mathbf{f}_{q}^{(\nu)}&(\mathbf{m}_{t})\big]\nonumber 
 \\ =\ & \bigg(   \int_{0}^{t}dr e^{(t-r) G^{(N)}_{R} } \big(G^{(N)}_{R}-G^{(\nu)}\big) e^{rG^{(\nu)} }   \mathbf{f}_{q}^{(\nu)}   \bigg)(\widehat{y}_{N})\,, \nonumber
\intertext{and since  $\mathbf{f}_q^{(\nu)}$ is an eigenfunction of $G^{(\nu)}$ with eigenvalue $-\frac{1}{2}|q|^{\nu+2}$  }
  =\ & \bigg(   \int_{0}^{t}dre^{-\frac{r}{2}|q|^{\nu+2}} e^{(t-r) G^{(N)}_{R} } \big(G^{(N)}_{R}-G^{(\nu)}\big)   \mathbf{f}_{q}^{(\nu)}   \bigg)(\widehat{y}_{N}) \,.\nonumber 
\intertext{Defining the set  $ \mathcal{S}_{N}:=\left\{ |y|\leq N^{\delta-\frac{\nu+1}{\nu+2}}  \right\}  $ for some $\frac{2\nu}{2\nu+1}\frac{\nu+1}{\nu+2} <\delta<\frac{\nu+1}{\nu+2}$ we can split this expression as follows:}
\ = \ &\underbrace{\bigg(   \int_{0}^{t}dre^{-\frac{r}{2}|q|^{\nu+2}} e^{(t-r) G^{(N)}_{R} } 1_{y \in  \mathcal{S}_{N} }\big(G^{(N)}_{R}-G^{(\nu)}\big)   \mathbf{f}_{q}^{(\nu)}   \bigg)(\widehat{y}_{N})}_{I} \nonumber \\  & + \bigg(   \int_{0}^{t}dre^{-\frac{r}{2}|q|^{\nu+2}} e^{(t-r) G^{(N)}_{R} } 1_{y\notin \mathcal{S}_{N} }\big(G^{(N)}_{R}-G^{(\nu)}\big)   \mathbf{f}_{q}^{(\nu)}   \bigg)(\widehat{y}_{N}) \  .\label{Gladrags}
\end{align}
 The insertion of  complementary indicator functions $1_{y\in  \mathcal{S}_{N} }$ and $1_{y\notin  \mathcal{S}_{N} }$ in the second equality above will help us avoid the singular behavior in the higher derivatives of $\mathbf{f}_{q}^{(\nu)}$ near zero when using Taylor expansions.   The second term on the right hand side of (\ref{Gladrags}) can be rewritten as the sum 
\begin{align}
- & \frac{|q|^{\nu+2}}{2}\underbrace{ \bigg(  \int_{0}^{t}dre^{-\frac{r}{2}|q|^{\nu+2}} e^{(t-r)G^{(N)}_{R} } 1_{y\notin  \mathcal{S}_{N} }\frac{Q_{N}(y)-Q(y)}{Q(y)} \mathbf{f}_{q}^{(\nu)}   \bigg)(\widehat{y}_{N})}_{II}  \nonumber\\ &+ \underbrace{\bigg(   \int_{0}^{t}dre^{-\frac{r}{2}|q|^{\nu+2}} e^{(t-r)G^{(N)}_{R} } 1_{y\notin  \mathcal{S}_{N}} E_{N} \bigg)(\widehat{y}_{N})}_{III}\ , \label{Honkytonk}
\end{align}
where $E_{N}: N^{-\frac{\nu+1}{\nu+2}} \widehat{Y}(\Z)\rightarrow \R$ is defined  so that the following relation holds 
\begin{align} E_{N}\bigg( \frac{ \widehat{Y}(n)}{  N^{\frac{\nu+1}{\nu+2}} }\bigg)\ = \ &\sum_{\pm}  \frac{NR_{n}^{\pm } }{2}\int_{ 0 }^{  \frac{ \widehat{Y}(n\pm 1)- \widehat{Y}(n)}{  N^{\frac{\nu+1}{\nu+2}} }   } dy\, y^{2} \, \frac{d^{3}\mathbf{f}_{q}^{(\nu)}}{dy^{3}}\bigg( \frac{ \widehat{Y}(n\pm 1)}{  N^{\frac{\nu+1}{\nu+2}} }-y\bigg)\ . \label{Clifford}
\end{align}
In the above, recall that for $\sigma=\textup{sgn}(n)$
$$\widehat{Y}(n\pm 1)\ - \ \widehat{Y}(n)\,=\,\sigma\frac{\nu+1}{R_{n}^{\pm\sigma}   }\,.   $$  To equate the last term of (\ref{Gladrags}) with~(\ref{Honkytonk}),   we have used \eqref{eq:welcometothejungle} and   second-order Taylor expansions of $ \mathbf{f}_{q}^{(\nu)}$ around $y=\frac{\widehat{Y}(n)}{  N^{\frac{\nu+1}{\nu+2}} }$ to obtain 
\begin{align}\label{BladerFail}
\big(G^{(N)}_{R} \mathbf{f}_{q}^{(\nu)} \big)\bigg( \frac{ \widehat{Y}(n)}{  N^{\frac{\nu+1}{\nu+2}} }\bigg)\  = \  & \frac{1}{2}Q_{N}\bigg( \frac{ \widehat{Y}(n)}{  N^{\frac{\nu+1}{\nu+2}} }\bigg)  \frac{d^{2}\mathbf{f}_{q}^{(\nu)}}{dy^{2}} \bigg( \frac{ \widehat{Y}(n)}{  N^{\frac{\nu+1}{\nu+2}} }\bigg)+E_{N}\bigg( \frac{ \widehat{Y}(n)}{  N^{\frac{\nu+1}{\nu+2}} }\bigg) \  .
 \intertext{Note that the first-order term is zero since $L_R \widehat{Y} \equiv 0$. 
  Moreover, $G^{(\nu)}\mathbf{f}_{q}^{(\nu)}=\frac{1}{2}Q(y)\frac{d^{2}\mathbf{f}_{q}^{(\nu)}}{dy^{2}}=-\frac{|q|^{\nu+2}}{2}\mathbf{f}_{q}^{(\nu)}$ by~(\ref{Habble}),  so the above is equal to }
   =\ &-\frac{|q|^{\nu+2}}{2}\frac{ Q_{N}\Big( \frac{ \widehat{Y}(n)}{  N^{\frac{\nu+1}{\nu+2}} }\Big) }{   Q\Big( \frac{ \widehat{Y}(n)}{  N^{\frac{\nu+1}{\nu+2}} }\Big)  } \mathbf{f}_{q}^{(\nu)} \bigg( \frac{ \widehat{Y}(n)}{  N^{\frac{\nu+1}{\nu+2}} }\bigg)+E_{N}\bigg( \frac{ \widehat{Y}(n)}{  N^{\frac{\nu+1}{\nu+2}} }\bigg)\ . \nonumber
\end{align}
  
We will bound the absolute values of $I$, $II$ and $III$ below.   

\vspace{.3cm}

\noindent $I$.   We have
\begin{align}\label{Locusts}
 |I| \ = \ &\bigg |  \bigg(   \int_{0}^{t}dr  e^{-\frac{r}{2}|q|^{\nu+2}} e^{(t-r)G^{(N)}_{R} }   1_{y \in  \mathcal{S}_{N} } \big(G^{(N)}_{R}-G^{(\nu)}\big)   \mathbf{f}_{q}^{(\nu)}   \bigg)(\widehat{y}_{N})  \bigg | \nonumber \\    \leq \ &\bigg(  \int_{0}^{t}dr \mathbb{P}_{\widehat{y}_{N}}\Big[ \big|M_{r}^{(N)}   \big|\leq  N^{\delta-\frac{\nu+1}{\nu+2}}  \Big] \bigg) \sup_{|\widehat{Y}(n) |\leq N^{\delta}   }\bigg|\Big(\big(G^{(N)}_{R}-G^{(\nu)}\big) \mathbf{f}_{q}^{(\nu)}\Big)\bigg( \frac{\widehat{Y}(n)   }{N^{\frac{\nu+1}{\nu+2}}    }\bigg)\bigg|   \nonumber \\ \leq  \   &  C \big( 1+t^{\frac{\nu+1}{\nu+2}}    \big) N^{\frac{\delta}{\nu+1}-\frac{1}{\nu+2}}\sup_{|\widehat{Y}(n) |\leq N^{\delta}   } \bigg( \bigg|\big(G^{(N)}_{R}\mathbf{f}_{q}^{(\nu)}\big)\bigg( \frac{\widehat{Y}(n)   }{N^{\frac{\nu+1}{\nu+2}}    }\bigg)\bigg|+\bigg|\big(G^{(\nu)}\mathbf{f}_{q}^{(\nu)}\big)\bigg( \frac{\widehat{Y}(n)   }{N^{\frac{\nu+1}{\nu+2}}    }\bigg)\bigg|\bigg)\ ,
\end{align}
where the second inequality above  applies part (2) of Lem.~\ref{SecMomentLem}.   We will show below that the supremum on the last line is uniformly  bounded for  $N>1$

The  term $\Big|\big(G^{(\nu)}\mathbf{f}_{q}^{(\nu)}\big)\Big( \frac{\widehat{Y}(n)   }{N^{\frac{\nu+1}{\nu+2}}    }\Big)\Big|$ in the supremum on the last line of~(\ref{Locusts}) is equal to  $\frac{1}{2}|q|^{\nu+2} \Big|\mathbf{f}_{q}^{(\nu)}\Big(  \frac{\widehat{Y}(n)   }{N^{\frac{\nu+1}{\nu+2}}    }\Big)\Big|$ by~(\ref{Habble}), which is close to  $\frac{1}{2}|q|^{\nu+2}$ for $|\widehat{Y}(n)    |\leq N^{\delta}$  and large $N$ since $\mathbf{f}_q^{(\nu)}(0)=1$.   The expression for  $\Big|\big(G^{(N)}_{R}\mathbf{f}_{q}^{(\nu)}\big)\Big( \frac{\widehat{Y}(n)   }{N^{\frac{\nu+1}{\nu+2}}    } \Big)\Big|$ can be rewritten using a first-order Taylor formula around $ y=\frac{\widehat{Y}(n)   }{N^{\frac{\nu+1}{\nu+2}}    } $ as follows
\begin{align*}
\bigg|\big(G^{(N)}_{R}\mathbf{f}_{q}^{(\nu)}\big)\bigg( \frac{\widehat{Y}(n)   }{N^{\frac{\nu+1}{\nu+2}}    } \bigg)\bigg|    =  \  &    \Bigg|  \sum_{\pm}N  R_{n}^{\pm }\int_{ 0 }^{  \frac{ \widehat{Y}(n\pm 1)- \widehat{Y}(n)}{  N^{\frac{\nu+1}{\nu+2}} }   } dy\, y \, \frac{d^{2}\mathbf{f}_{q}^{(\nu)}}{dy^{2}}\bigg( \frac{ \widehat{Y}(n\pm 1)}{  N^{\frac{\nu+1}{\nu+2}} }-y\bigg) \Bigg| \ ,   \\  \intertext{where again the first-order terms cancel since $\widehat{Y}(X_{t})$ is a martingale.  This expression is, in turn, bounded by}    \leq \  &   C\delta_{0}(n)+ C\big(1-\delta_{0}(n)\big)\bigg( \frac{1}{| R_{n}^{+ }   |}+  \frac{1}{| R_{n}^{- }   |} \bigg) \big|  \widehat{Y}(n)   \big|^{-\frac{\nu}{\nu+1}}   \, \leq \,C \,,
\end{align*}
where the first inequality above follows because $\big| \frac{d^{2}\mathbf{f}_{q}^{(\nu)}}{dy^{2}}\big|$  is bounded by a constant multiple of $|y|^{-\frac{\nu}{\nu+1}}$ over the domain $|y|\leq 1$ by  part (2) of Prop.~\ref{PropEigenKets}.   The second inequality holds  since $\frac{1}{| R_{n}^{+ }   |}+  \frac{1}{| R_{n}^{- }   |} $ is bounded by a constant multiple of  $|\widehat{Y}(n)  |^{\frac{\nu}{\nu+1}}$ for all $n\neq 0$ as a consequence of the asymptotics \eqref{JumpAsymptotics}.

Thus $I$ is $ \mathcal{O}\big(  N^{\frac{\delta}{\nu+1}-\frac{1}{\nu+2}} \big)  $ and decays for large $N$ by our assumption that $\delta<\frac{\nu+1}{\nu+2}$.

\vspace{.3cm}

\noindent $II$.  
By part (2) of Lem.~\ref{SecLemMisc}, we can bound  $\Big| \frac{ Q_{N}(y)-Q(y)  }{ Q(y)  }  \Big|  $ by  $CN^{-\frac{1}{\nu+2}}|y|^{-\frac{1}{\nu+1}}$ for $y\neq0$.  Thus the absolute value of $II$ is smaller than 
\begin{align*}
N^{-\frac{\delta}{\nu+1} }   \int_{0}^{t}dr e^{-\frac{r}{2}|q|^{\nu+2}}  \int_{\R }\Big(  e^{(t-r) G^{(N)}_{R} }\Big)(\widehat{y}_{N},y)  \big| \mathbf{f}_{q}^{(\nu)}(y)  \big|  \, \leq \,  & N^{-\frac{\delta}{\nu+1} }  \int_{0}^{t} dr\mathbb{E}_{\widehat{y}_{N}}\Big[\big| \mathbf{f}_{q}^{(\nu)}\big(M^{(N)}_{r}\big)\big| \Big] \ .   \nonumber  \\ 
 \intertext{Recall that  $\big| \mathbf{f}_{q}^{(\nu)}(y) \big|$ is bounded by a constant multiple of $1+|y|^{\frac{\nu}{4(\nu+1)}}$.   Moreover,   $\big|M^{(N)}_{r}\big|$ is bounded by a constant multiple of $\big| X^{(N)}_{r} \big|^{\nu+1}$ since $M^{(N)}_{r}=Y_{N}\big(X^{(N)}_{r}\big)$ and $\big|Y_{N}(x)\big|\leq C|x|^{\nu+1}$ by part (4) of  Lem.~\ref{SecLemMisc}.  Thus the above is bounded by the following:}
\nonumber \leq \, & CtN^{-\frac{\delta}{\nu+1} } \mathbb{E}_{\widehat{x}_{N}}\bigg[1+\sup_{0\leq r\leq t} \big|X^{(N)}_{r}\big|^{\frac{\nu}{4} }   \bigg]   \nonumber  \\   \leq \, &   Ct N^{-\frac{\delta}{\nu+1} }\Big(1+ |\widehat{x}_{N}|^{\frac{\nu}{4} }+t^{\frac{\nu}{4(\nu+2)} } \Big)\ .
\end{align*} 
  The last inequality is by part (1) of Lem.~\ref{SecMomentLem}.   Since $ \widehat{x}_{N}$ converges  to $\widehat{x}$  for large  $N$ and thus is a bounded sequence, we have that $II$ is $\mathcal{O}\big(  N^{-\frac{\delta}{\nu+1} }  \big)$.

\vspace{.3cm}

\noindent $III$. By applying part (2) of Prop.~\ref{PropEigenKets} to bound the values of  $\frac{d^{3}\mathbf{f}_{q}^{(\nu)}}{dy^{3}}$ in the formula~(\ref{Clifford}), we have that for all  $n\in\Z$ and $N>1$ with $| \widehat{Y}(n)|>N^{\delta  } $,
\begin{align}\label{Trayvon}
\bigg|E_{N}\bigg(\frac{   \widehat{Y}(n)  }{  N^{\frac{\nu+1}{\nu+2}}  }      \bigg)   \bigg|  \leq  C N^{-\delta \frac{ 2\nu+1  }{\nu+1} } \bigg( \frac{  1 }{  ( R_{n }^{+}  )^{2}}+\frac{1 }{ (R_{n}^{-})^{2}  }\bigg)\ .
\end{align}
Thus we have the following bounds:
\begin{align*}
|III|  \ = \ \bigg|  \bigg(   \int_{0}^{t}dre^{-\frac{r}{2}|q|^{\nu+2}} e^{(t-r) G^{(N)}_{R} }& 1_{y\notin  \mathcal{S}_{N}} E_{N} \bigg)(\widehat{y}_{N})  \bigg| \\   \leq \ &  t  \sup_{r\in [0,t] }\mathbb{E}_{\widehat{y}_{N}}\bigg[\left|E_{N}\big( M_{r}^{(N)}  \big)    \right|  \chi\Big(\big| M_{r}^{(N)}\big|>N^{\delta-\frac{\nu+1}{\nu+2}   }\Big) \bigg] \\  \leq \ & C t N^{-\delta \frac{ 2\nu+1  }{\nu+2}} \sup_{r\in [0,Nt] }\mathbb{E}_{x_{N}}\bigg[ \frac{  1 }{  ( R_{X_{r} }^{+}  )^{2}}+\frac{1 }{ (R_{X_{r}}^{-})^{2}  }\bigg]\\ \leq \ &  CtN^{-\delta \frac{ 2\nu+1  }{\nu+1}}\bigg(1+\mathbb{E}_{x_{N}}\bigg[  \sup_{r\in [0,Nt] }\big| X_{r}  \big|^{2\nu} \bigg]\bigg) \\  = \ & \mathcal{O}\Big(  N^{\frac{2\nu}{\nu+2}-\delta\frac{ 2\nu+1  }{\nu+1}} \Big)\ .
\end{align*}
The first inequality above uses that  $\big(e^{rG^{(N)}_{R} } E_{N} \big)(\widehat{y}_{N}) =\mathbb{E}_{\widehat{y}_{N}}\big[ E_{N} \big(  M_{r}^{(N)} \big)  \big]$, the second inequality is by~(\ref{Trayvon}),  
  and the third inequality holds by the asymptotic assumption~(\ref{Oreilly}) on the jump rates $R_{n}^{\pm}$.  The order equality is by part (1) of Lem.~\ref{SecMomentLem}.   The last line is decaying for large $N$ by our assumption that $\delta>\frac{2\nu}{2\nu+1}\frac{\nu+1}{\nu+2}$.  

\vspace{.3cm}

\noindent (\textup{iii}).   For the third term on the right side of~(\ref{Zing}), we can use that  $  \mathbf{f}_{q}^{(\nu)}  $ is an eigenvector of $G^{(\nu)}$ again 
\begin{align*}
\Big|  \mathbb{E}_{\widehat{y}_{N}}\big[\mathbf{f}_{q}^{(\nu)}(\mathbf{m}_{t})\big]   -  \mathbb{E}_{\widehat{\mathbf{y}} }\big[\mathbf{f}_{q}^{(\nu)}(\mathbf{m}_{t})\big] \Big|\  = \  &\Big| \big( e^{tG^{(\nu)}}\mathbf{f}_{q}^{(\nu)} \big)(\widehat{y}_{N})  -\big(  e^{tG^{(\nu)}}\mathbf{f}_{q}^{(\nu)} \big)(\widehat{\mathbf{y}})  \Big| \\ =\ & e^{-\frac{t}{2}|q|^{\nu+2}}\Big|  \mathbf{f}_{q}^{(\nu)}(\widehat{y}_{N})   -  \mathbf{f}_{q}^{(\nu)}(\widehat{\mathbf{y}}) \Big|  \\ \leq \ & \big|\widehat{y}_{N}-\widehat{\mathbf{y}} \big|\sup_{y\in \R}\bigg| \frac{d   \mathbf{f}_{q}^{(\nu)}}{dy}(y)\bigg|  =  \mathit{o}(1    )\ .
\end{align*}
The difference $ \widehat{y}_{N}-\widehat{\mathbf{y}}=Y_{N}( \widehat{x}_{N})-Y(\widehat{\mathbf{x}})$ converges to zero since $Y(x)$ is continuous and  $Y_{N}(x)$ converges uniformly to $Y(x)$ over compact sets as a result of part (1) of Lem.~\ref{SecLemMisc}.  As remarked above the derivative of  $\mathbf{f}_{q}^{(\nu)}$ is uniformly bounded.  Thus, all the terms on the right side of~(\ref{Zing}) vanish for large $N$ and the convergence of the one-dimensional distributions is established.  \vspace{.3cm}

We have proved convergence of the processes $X_t^{(N)}$ to $\mathbf{x}_t$ at a single time.  More precisely, we have proved weak convergence of the transition measures $\phi^{(N)}_t(\widehat{x}_N,x')dx'$ for $X_t^{(N)}$ to $\phi_t(\widehat{\mathbf{x}},x')dx'$ whenever $\widehat{x}_N\rightarrow \widehat{\mathbf{x}} $.   However, by part (3) of Lem.~\ref{SecMomentLem} the family $X_t^{(N)}$ with $X_0^{(N)}=\widehat{x}_N$ and $N>1$ is tight.  Thus there are subsequences $N_j \rightarrow \infty$ such that the processes converge in law to \emph{some} limit process.  To complete the proof it suffices to prove that any such subsequential limit process is $\mathbf{x}_t$ with  $\mathbf{x}_0 =\widehat{\mathbf{x}}$. However, any limit process is necessarily Markovian and since its transition measures are $\phi_t(x,x')dx'$ it must be $\mathbf{x}_t$.  \qed

\subsection{Proofs of the technical lemmas}\label{SecDiffTech}

We begin with Lem.~\ref{SecLemMisc}.
\begin{proof}[Proof of Lemma~\ref{SecLemMisc}]\text{  }

\noindent Part (1): Since $R_{n}^{\pm}>0$ is bounded away from zero  on finite subsets of $\Z$    and has the limiting form (\ref{Oreilly}), we see that for all $n\in \Z$ and  $x=nN^{-\frac{1}{\nu+2}}$
\begin{align}\label{Chimps} 
\bigg|\frac{1}{2N^{\frac{\nu}{\nu+2}}R_{n}^{\pm}}- |x|^{\nu}\bigg| \ \leq & \ C N^{-\frac{1}{\nu+2}}   |x|^{\nu-1}\ .
\end{align}
Let us assume $x>0$.   For all $x\in N^{-\frac{1}{\nu+2}}\Z$ and $N>1$, we have the relations
\begin{align*}
\big|  Y_{N}(x) -Y(x)  \big|  \ = & \  (\nu+1)  \Bigg| \,N^{-\frac{1}{\nu+2}}\sum_{n=1}^{ N^{\frac{1}{\nu+2}} x } \frac{1}{2N^{\frac{\nu}{\nu+2}}}\frac{1}{R_{n}^{-}}- \int_{0}^{x}da  |a|^{\nu}\Bigg| \\  \leq & \ C N^{-\frac{1}{\nu+2}}  \int_{0}^{x}da \,a^{\nu-1}
    \\  = & \ C N^{-\frac{1}{\nu+2}}|x|^{\nu}\ ,
\end{align*}
where the  inequality uses a Riemann sum approximation and~(\ref{Chimps}).

\vspace{.3cm}

\noindent Part (2): Define $W(y):=\sgn(y)|y|^{\frac{1}{\nu+1}}$, i.e., the function inverse of $Y$.   First we will show 
that for  $x\in N^{-\frac{1}{\nu+2}}\Z$ 
\begin{align}\label{Sylow}
\big|W\big( Y_{N}(x)  \big)-x\big| \ \leq  \  C N^{-\frac{1}{\nu+2}}\ . 
\end{align}
Note that the left-hand side is zero for $x=0$.  For $x\neq 0$ we have $\sgn(Y_N(x))=\sgn(Y(x))=\sgn(x)$. Since $W$ is concave on $(-\infty,0)$ and $(0,\infty)$ we may bound the difference between $W\big( Y_{N}(x)  \big)$ and $x$ 
as 
\begin{align*}
\big|W\big( Y_{N}(x)  \big)-x\big| \ = \ \big|W\big( Y_{N}(x)  \big)-W \big( Y(x) \big )\big| \ \leq & \ \frac{1}{\nu+1} |Y(x)|^{-\frac{\nu}{\nu+1}} \,  \big | Y_N(x) - Y(x) \big |\\  \, = & \ C |x|^{-\nu} \big | Y_N(x) - Y(x) \big | \ \\
 \leq & \  C N^{-\frac{1}{\nu+2}} \  .
\end{align*}
The second inequality applies part (1). 

 \vspace{.2cm}

To bound $|Q_N(y) -Q(y)|$, note that $Q(y)=(\nu+1)^2 \big | W(y)|^\nu$. 
Define  $W_{N}(y):=N^{-\frac{1}{\nu+2}}\widehat{W}\big(N^{\frac{\nu+1}{\nu+2} } y\big)    $ and $n:=N^{\frac{1}{\nu+2}}W_{N}(y)   $.     Notice that~(\ref{Sylow}) implies    
\begin{equation}\label{Sylow2}
\big |W(y) - W_N(y) \big | \ \le \ C N^{-\frac{1}{\nu+2}}\ .
\end{equation}
  By the triangle inequality and~\eqref{Chimps}, we get the inequalities 
\begin{align}\nonumber
\big|   Q_{N}(y)-Q(y)   \big|  \ \leq &  \ C \bigg(\Big| \frac{1}{ N^{\frac{\nu}{\nu+2}} R_{n }^+}+\frac{1}{ N^{\frac{\nu}{\nu+2}} R_{ n }^-}-\big|W_{N}(y) \big|^{\nu}       \Big|+\Big|   \big|W_{N}(y) \big|^{\nu} -\big|W(y) \big|^{\nu}       \Big| \bigg)\\
\leq &  \   C \left ( N^{-\frac{\nu}{\nu+2}} + N^{-\frac{1}{\nu+2}} \big |W_{N}(y)\big|^{\nu-1}  +
\Big|   \big|W_{N}(y) \big|^{\nu} -\big|W(y) \big|^{\nu} \Big| \right ).\nonumber
  \intertext{Applying~(\ref{Sylow2})  gives the further bound}
 \leq & \ C \left [ N^{-\frac{\nu}{\nu+2}} + N^{-\frac{1}{\nu+2}} \big |W_{N}(y)\big|^{\nu-1}  + N^{-\frac{1}{\nu+2}} \max \left ( |W_{N}(y)|^{\nu-1}, |W(y)|^{\nu-1} \right ) \right ]\,.\nonumber
\end{align}
 The inequality \eqref{Sylow2} also implies that $c |W(y)| \le |W_{N}(y)|\le C |W(y)|$, and thus
$$ \big|   Q_{N}(y)-Q(y)   \big| \ \le CN^{-\frac{1}{1+\nu}}|y|^{\frac{\nu-1}{\nu+1}} $$
for  $y\neq 0$ and $N>1$ as claimed.  
\vspace{.3cm}

\noindent Part (3):  
Using two first-order Taylor expansions of $f(y)=|y|^{\frac{\nu+2}{\nu+1}}$ around $y= \widehat{Y}(n)  $, we can write $\widehat{A}(n)$ in the form
\begin{align*}
\widehat{A}(n)\,=\,& \frac{\nu+2}{(R_{n}^+)^2} \int_{0}^{1} dy \big(1- y \big)  \Big( y\frac{\nu+1}{R_{n}^+} +  \widehat{Y}(n) \Big)^{-\frac{\nu}{\nu+1}} \\ &\,+\,\frac{\nu+2}{(R_{n}^-)^2 }\int_{0}^{1} dy \big(1- y \big)  \Big(-y\frac{\nu+1}{R_{n}^-} +  \widehat{Y}(n) \Big)^{-\frac{\nu}{\nu+1}} \,.
\end{align*}
The values of $\widehat{A}(n)$ are strictly positive, and  for large $n$ it follows from the asymptotics \eqref{JumpAsymptotics} that
\begin{align*}
\widehat{A}(n)\,\approx \,   \bigg(\frac{\nu+2}{(R_{n}^+)^2} +\frac{\nu+2}{(R_{n}^-)^2 }\bigg)  \big( \widehat{Y}(n)   \big)^{-\frac{\nu}{\nu+1}} \,=\, 8(\nu+2)\,+\,\mathcal{O}\Big(\frac{1}{n}\Big)\,.
\end{align*}
Hence $\widehat{A}(n)$ is bounded away from zero and $A_{N}(y)$ is also. \vspace{.3cm}

\noindent Part (4): These inequalities follow from parts (1) and (2).

\end{proof}\vspace{.5cm}

Before going  into the proof of  Lem.~\ref{SecMomentLem}, we state the following lemma, which  bounds the size of the jumps of the martingale $\big(M_{r}^{(N)}\big)_{r\geq 0}$.

\begin{lemma}\label{LemJumpSize}
For $r\in \R^+$, define  $\Delta_{r}^{(N)}:=  M_{r^+}^{(N)} -M_{r^-}^{(N)}      $.   There is a $C>0$ such that for all $t>0$ and $N>1$ 
$$   \sup_{0\leq r\leq t}\big| \Delta_{r}^{(N)}\big|  \,\leq \, \frac{C}{ N^{\frac{1}{\nu+1}}}\sup_{0\leq r \leq t}  \big| M_{r}^{(N)}\big|^{\frac{\nu}{\nu+2}}\,.$$

\end{lemma}

\begin{proof} Recall that $M_{r}^{(N)}:= N^{-\frac{\nu+1}{\nu+2}}\widehat{Y}(X_{Nt})  $ and that 
$$\widehat{Y}(n )  \, =\, \sigma(\nu +1)\sum_{k=1}^{n}\frac{1}{ R_{k}^{-\sigma} }\,,$$
where $\sigma\in \{\pm \}$ is the sign of $n\in \Z$.  Since $R_{k}^{+}=R_{k+1}^{-}$,  the jumps $\big|\Delta_{r}^{(N)}\big|$ of $M_{r}^{(N)}$ have the form $ N^{-\frac{\nu+1}{\nu+2}}\frac{\nu+1}{ R_{n}^{\pm} }$ for $n=X_{Nr}$.  The result follows from the asymptotic formula~(\ref{JumpAsymptotics}), which implies that
$$\widehat{Y}(n ) \,=\,  \frac{1}{2}n^{1+\nu}\,+\,\mathit{O}(n^{\nu})\,.$$

\end{proof}\vspace{.5cm}

\begin{proof}[Proof of Lemma~\ref{SecMomentLem}]\text{  }\\
Part (1):  By Jensen's inequality we have the first inequality below: 
\begin{align}\label{Hassel}
  \mathbb{E}_{x}\bigg[ \sup_{0\leq r\leq t} \big|X_{r}^{(N)}\big|^{n}   \bigg] \ \leq \  &  \mathbb{E}_{x}\bigg[ \sup_{0\leq r\leq t} \big|Y\big(X_{r}^{(N)}\big)\big|^{2n}   \bigg]^{\frac{1}{2\nu+2}} \   \leq  \  C \mathbb{E}_{y}\bigg[ \sup_{0\leq r\leq t}  \big| M_{r}^{(N)}\big|^{2n}   \bigg]^{\frac{1}{2\nu+2}} \  ,
\end{align}
where $y:=Y_{N}(x)$.  The second inequality holds since $M_{t}^{(N)}:= Y_{N}(X_{t}^{(N)})$ and $Y(x)$ is bounded by a constant multiple of $Y_{N}(x)$; see part (4) of Lem.~\ref{SecLemMisc}.    Also as a consequence of part (4) of Lem.~\ref{SecLemMisc}, 
\begin{align}\label{Blimp}
     |y|^{\frac{1}{\nu+1}} \ = \  \big| Y_{N}(x)  \big|^{\frac{1}{\nu+1}}\ \leq \  C|x| \ .
\end{align}
With~(\ref{Hassel}) and~(\ref{Blimp}), it sufficient to show that $ \mathbb{E}_{y}\big[ \sup_{0\leq r\leq t}  \big| M_{r}^{(N)}\big|^{2n}   \big]$ is bounded by a constant multiple of $ 1+|y|^{2n}+t^{ \frac{2n(\nu+1)}{\nu+2}}  $ for all $t\in \R^{+}$ and $N>1$ since
$|y|^{\frac{1}{\nu+1}}=\big| Y_{N}(x)  \big|^{\frac{1}{\nu+1}}$ is bounded by a constant multiple of $|x|$.  
  Applying Doob's maximal inequality  to the submartingale $\big|M_{r}^{(N)}\big|^{2}$ gives us the first inequality below:    
\begin{align}\label{CherTwo}
 \mathbb{E}_{y}\bigg[ \sup_{0\leq r\leq t} & \big| M_{r}^{(N)}\big|^{2n}   \bigg] \nonumber \\  \  \leq \ & C \mathbb{E}_{y}\Big[ \big| M_{t}^{(N)}\big|^{2n}   \Big]\, =\,\mathbb{E}_{y}\bigg[\Big|y+\int_{0}^{t}d M_{r}^{(N)}   \Big|^{2n}  \bigg] \nonumber  \\  \ \leq  \ &  C|y|^{2n}\,+\,  C\mathbb{E}_{y}\bigg[\Big|\int_{0}^{t}d M_{r}^{(N)}   \Big|^{2n}  \bigg]\,,\nonumber   
\intertext{where the second inequality is simply  $(a+b)^{2n}\leq 4^{n}(a^{2n}+b^{2n})$.  Define $\Delta_{r}^{(N)}:=  M_{r^+}^{(N)} -M_{r^-}^{(N)}      $.   By Rosenthal's inequality,    }
\leq \  &  C|y|^{2n}+C\mathbb{E}_{y}\bigg[\Big|\int_{0}^{t}dr Q_{N}\big(  M_{r}^{(N)} \big)  \Big|^{n}  \bigg]\,+\,C\mathbb{E}_{y}\bigg[\sup_{0\leq r\leq t}\big| \Delta_{r}^{(N)}  \big|^{2n} \bigg] \,. \nonumber
\intertext{We can bound the second term above using Lemma~\ref{LemJumpSize}:}
  \leq \  &  C|y|^{2n}+Ct^{n-1}\mathbb{E}_{y}\bigg[\int_{0}^{t}dr\big|Q_{N}\big(  M_{r}^{(N)} \big) \big|^{n}    \bigg]\, +\,\frac{C}{N^{\frac{2n}{\nu+1}}}\mathbb{E}_{y}\bigg[\sup_{0\leq r\leq t}\big| M_{r}^{(N)}  \big|^{\frac{2n\nu}{\nu+1} } \bigg]  \ . \nonumber 
 \intertext{By  part (4) of Lem.~\ref{SecLemMisc}, the above is smaller than }
 \leq \  &C |y|^{2n}+ Ct^{n-1}\mathbb{E}_{y}\bigg[\int_{0}^{t}dr\Big(1+  \big| M_{r}^{(N)}\big|^{\frac{n\nu}{\nu+1}} \Big)\bigg]\, +\,\frac{C}{N^{\frac{2n}{\nu+1}}}\mathbb{E}_{y}\bigg[\sup_{0\leq r\leq t}\big| M_{r}^{(N)}  \big|^{\frac{2n\nu}{\nu+1} } \bigg]  \nonumber   \\ \leq \ &C |y|^{2n}+  Ct^{n}\bigg(1+ \mathbb{E}_{y}\bigg[\sup_{0\leq r \leq t}  \big| M_{r}^{(N)}\big|^{\frac{n\nu}{\nu+1}} \bigg]\bigg)\, +\,\frac{C}{N^{\frac{2n}{\nu+1}}}\mathbb{E}_{y}\bigg[\sup_{0\leq r\leq t}\big| M_{r}^{(N)}  \big|^{\frac{2n\nu}{\nu+1} } \bigg] \nonumber \\   \leq \  &C |y|^{2n}+  Ct^{n}\bigg(1+ \mathbb{E}_{y}\bigg[\sup_{0\leq r \leq t}  \big| M_{r}^{(N)}\big|^{2n} \bigg]^{\frac{\nu}{2(\nu+1)}}\bigg)\, +\,\frac{C}{N^{\frac{2n}{\nu+1}}}\mathbb{E}_{y}\bigg[\sup_{0\leq r\leq t}\big| M_{r}^{(N)}  \big|^{2n } \bigg]^{\frac{\nu}{\nu+1}}\ .
\end{align}
The last inequality  is Jensen's. 

For large enough $N$, \eqref{CherTwo} implies that
$$  \mathbb{E}_{y}\bigg[ \sup_{0\leq r\leq t}  \big| M_{r}^{(N)}\big|^{2n}   \bigg]\,\leq \,C\,+\,C |y|^{2n}+  Ct^{n}\bigg(1+ \mathbb{E}_{y}\bigg[\sup_{0\leq r \leq t}  \big| M_{r}^{(N)}\big|^{2n} \bigg]^{\frac{\nu}{2(\nu+1)}}\bigg)\,.$$
As it stands, the above  holds trivially if  $ \mathbb{E}_{y}\Big[ \sup_{0\leq r\leq t}  \big| M_{r}^{(N)}\big|^{2n}   \Big] $ were $\infty$. However, we may replace $M_r^{(N)}$ by a the martingale 
$M_r^{(N,L)} := M_{r\wedge \tau_L}^{(N)}$ where $\tau_L$ is the first time that $|M_r^{(N)}|=L$. The same reasoning as above shows that for $U_{N,L,t}:=\mathbb{E}_{y}\Big[\sup_{0\leq r\leq t}\big|M_{r\wedge\tau_L}^{(N)}\big|^{2n} \Big]$
\begin{align}\label{Dink}
U_{N,L,t}\  \leq \  C+ C  |y|^{2n}+  C t^{n} +  C t^{n} U_{N,L,t}^{\frac{\nu}{2(\nu+1)}}
\end{align}
for all $y\in \R$, $t\in \R^{+}$  and $N>1$ \emph{with a constant that is uniform in $L>0$.}  Multiplying and dividing by an arbitrary $\lambda>0$ in the last term of~(\ref{Dink}) and applying Young's inequality, we find that 
$$ U_{N,L,t} \ \leq \ C  \Big ( 1+|y|^{2n} +  t^{n}  +  \frac{\nu +2}{2(\nu+1)} \lambda^{-\frac{2(\nu+1)}{\nu+2}}  t^{\frac{2n(\nu+1)}{\nu+2}} \Big ) + C \frac{\nu}{2(\nu+1)}\lambda^{\frac{2(\nu+1)}{\nu}} U_{N,L,t} \ .$$
Since $U_{N,L,t} \le L^{2n}$ is finite, we conclude by choosing $\lambda$ sufficiently small that in fact $U_{N,t}$ is uniformly bounded by a multiple of $1+|y|^{2n}+t^{\frac{2n(\nu+1)}{\nu+2}}$.  Taking $L\rightarrow \infty$ yields the desired bound on $ \mathbb{E}_{y}\Big[ \sup_{0\leq r\leq t}  \big| M_{r}^{(N)}\big|^{2n}   \Big] $.\vspace{.5cm}

\noindent Part (2):  We can write the expression that we wish to bound as  follows:
 $$        \int_{0}^{t}dr   \mathbb{P}_{y}\big[\big|M_{r}^{(N)}\big|\leq  N^{-\alpha}    \big]\ =\  \mathbb{E}_{y}\big[ \mathbf{T}_{ t}^{(N)}  \big] \hspace{.5cm}\text{for}\hspace{.5cm} \mathbf{T}_{t}^{(N)}\ :=\ \int_{0}^{t}dr \chi\big( \big|M_{r}^{(N)}\big|\leq  N^{-\alpha}   \big) \ . $$
In words $\mathbf{T}_{t}^{(N)}$ is the amount of time that the process $|M_{r}^{(N)}|$ spends below $N^{-\alpha}   $ over the interval $[0,t]$.  If the initial value $y$ is greater than $N^{-\alpha}$ it is clear that $\mathbb{E}_y\big[T_t^{(N)}\big] \le \mathbb{E}_{N^{-\alpha}}\big[T_t^{(N)}\big]$ and similarly for $y <-N^{-\alpha}$.  Hence we may assume without loss of generality that $|y| \le N^{-\alpha}$.

 It will be useful to partition the trajectory of $ \big(M_{r}^{(N)}\big)_{r\geq 0} $ into a series of incursions and excursions from the set $ |y|\leq N^{-\alpha} $.   Set $\varsigma_{0}=\varsigma^{\prime}_{1}=0$, and define the stopping times $\varsigma_{j},\varsigma_{j}^{\prime}$ such that for $j\geq 1$,
\begin{align*}
\varsigma^{\prime}_{j}= \min \Big\{r\in [\varsigma_{j-1},\infty)\,\Big| \,\big|M_{r}^{(N)}\big|\leq  N^{-\alpha}     \Big\}\quad \text{and} \quad
\varsigma_{j}= \min \Big\{ r\in [\varsigma^{\prime}_{j},\infty)\,\Big| \, \big|M_{r}^{(N)}\big|\geq 2N^{-\alpha}          \Big\} \ .
\end{align*}
The above definition uses that $ \big|M_{0}^{(N)}\big|\leq N^{-\alpha}  $ as otherwise we should begin only with  $\varsigma_{0}=0$.  
Let $\mathbf{n}_{t}$ be the number $\varsigma^{\prime}_{j}$'s for $j\geq 1$ less than $t$.  In other words, $\mathbf{n}_{t}$ is  the number of up-crossings of $ \big|M^{(N)}_{r}\big|$ from $  N^{-\alpha}      $ to $2 N^{-\alpha}   $  that have been completed or begun by time $t$. 
The definitions give us the  inequality $$\mathbf{T}_{t}\ \leq \ \sum_{j=1}^{\mathbf{n}_{t}}\varsigma_{j}-\varsigma^{\prime}_{j}\ .    $$
Next observe that
\begin{align}\label{IncursionTime}
\mathbb{E}_{y}\big[\mathbf{T}_{t}  \big]\ \leq \ \mathbb{E}_{y}\Bigg[  \sum_{j=1}^{\mathbf{n}_{t}}\varsigma_{j}-\varsigma^{\prime}_{j}\Bigg]\ \leq \ \mathbb{E}_{y}\big[ \mathbf{n}_{t}\big]\sup_{j\in \mathbb{N}}\mathbb{E}\big[\varsigma_{j}-\varsigma^{\prime}_{j}\,\big|\, j\leq \mathbf{n}_{t }   \big]  \ . 
\end{align}
With the above, we have an upper bound for $\mathbb{E}_{y}\big[\mathbf{T}_{t}  \big]$  in terms of the expectation of the number of up-crossings $\mathbf{n}_{t}$ and the expectation for the duration of a single up-crossing $ \varsigma_{j}-\varsigma_{j}^{\prime}$ conditioned on the event $j\leq \mathbf{n}_{t}$.   By the submartingale up-crossing inequality~\cite[Thm. 1.3.8]{Karat}, we have the first inequality below:
\begin{align}
 \mathbb{E}_{y}\big[\mathbf{n}_{t} \big]  \   \leq \  \frac{ \mathbb{E}_{y}\big[\big|M_{t}^{(N)}\big|  \big]+N^{-\alpha}}{2N^{-\alpha}    - N^{-\alpha}       } \  \leq  \ C N^{\alpha} \mathbb{E}_{y}\Big[1+\big|X_{t}^{(N)}\big|^{\nu+1}  \Big] \  \leq  \  C N^{\alpha} \big( 1+t^{\frac{\nu+1}{\nu+2}}   \big)\ . \label{Ying}
\end{align}
The second inequality holds by part (4) of  Lem.~\ref{SecLemMisc} since   $M_{t}^{(N)}=Y_{N}\big( X_{t}^{(N)}  \big) $,   and the third inequality is by part (1) above.

We now focus on the expectation of the incursion lengths $\varsigma_{j}-\varsigma_{j}^{\prime}$ appearing in~(\ref{IncursionTime}).  Whether or not the event $j\leq \mathbf{n}_{t}$ occurred will be known at time $\varsigma^{\prime}_{j}$, so the strong Markov property implies that   
$$\sup_{j\in \mathbb{N}}\mathbb{E}\big[\varsigma_{j}-\varsigma_{j}^{\prime}\,\big|\, j\leq \mathbf{n}_{t}   \big]\ \leq \  \sup_{|a|\leq  N^{-\alpha}   }\mathbb{E}_{a}[\varsigma_{1} ] \ . $$
Moreover, we can apply an argument similar to that leading to~(\ref{UseSubMart}) to bound the expectation of the stopping time $\varsigma_{1}$.  Recall from Prop.~\ref{SecLemMisc} that $S_{t}^{(N)}:=\big|M_{t}^{(N)}\big|^{\frac{\nu+2}{\nu+1}}$ is a submartingale for which the increasing part of its Doob-Meyer decomposition is given by $\int_{0}^{t}drA_{N}(S_{r}) $. The value of $\big|M_{r}^{(N)}\big|$ at the time $r=\varsigma_{1}$ has the bound 
\begin{align*}
\big|M_{\varsigma_{1}}^{(N)}\big| \,\leq \,& 2N^{-\alpha}\,+\,\big|\Delta_{\varsigma_{1}}^{(N)}\big| \,,
\intertext{where $\big|\Delta_{\varsigma_{1}}^{(N)}\big|$ is the size of the last  jump of $M_{r}^{(N)}$ out of the set $\big\{|y|\leq 2N^{-\alpha}\big\}$.    By similar reasoning as Lemma~\ref{LemJumpSize}, we can bound $\big|\Delta_{\varsigma_{1}}^{(N)}\big|$ using the value of  $M_{\varsigma_{1}- }^{(N)}$  } 
  \,\leq \, & 2N^{-\alpha}\,+\,  CN^{-\frac{1}{\nu+1}}\big| 2N^{-\alpha}    \big|^{\frac{\nu}{\nu+1}}\\
 \,\leq \, &  CN^{-\alpha}\,.
\end{align*}
The last inequality holds by our assumption that $\alpha\leq 1$.
The above gives us the first equality below:
\begin{align}\label{Gimps}
 C^{\frac{\nu+2}{\nu+1}}     N^{-\alpha\frac{\nu+2}{\nu+1}}  \ \geq \ \mathbb{E}_{a}\Big[\big|M_{\varsigma_{1}}^{(N)}\big|^{\frac{\nu+2}{\nu+1}}\Big] \  = \ \mathbb{E}_{a}\bigg[\int_{0}^{\varsigma_{1}}drA_{N}\big(M_{r}^{(N)}\big)    \bigg]\ \geq \ c^{-1} \mathbb{E}_{a}[ \varsigma_{1}] \ .
\end{align}
The equality in~(\ref{Gimps}) follows from   the optional stopping theorem.   For the third inequality, we apply part (3) of Lem.~\ref{SecLemMisc} to get a uniform lower bound for  $A_{N}(y)$.  

Applying the results~(\ref{Ying}) and~(\ref{Gimps}) in~(\ref{IncursionTime}), we have that for all $t\in \R^{+}$, $y\in \R$, $\alpha\leq 1$, and $N>1$ $$\mathbb{E}_{y}\big[\mathbf{T}_{t}  \big]\  \leq \  C  N^{\frac{-\alpha}{\nu+1}} \big( 1+t^{\frac{\nu+1}{\nu+2}}   \big) \ .\qedhere$$

\vspace{.5cm}
\noindent Part (3):  Since $Y:\R\rightarrow \R$  has a continuous inverse, it is sufficient to show that the family of processes $\big(Y\big(X^{(N)}_{t} \big); t\in [0,T]\big)$ with $N>1$ is tight.  Moreover, it is sufficient to prove tightness for  $M^{(N)}_{t}=Y_{N}\big(X^{(N)}_{t}\big) $; to see this note that  $\sup_{t\in [0,T]} \big | Y\big(X^{(N)}_{t}\big)-Y_{N}\big(X^{(N)}_{t}\big) \big |$ converges to zero in probability as $N\rightarrow \infty$ because by part (1) of Lem.~\ref{SecLemMisc} we have the inequality
\begin{align*}
 \mathbb{E}_{\widehat{x}_{N}}\bigg[\sup_{0\leq t\leq T}\Big| Y\big(X^{(N)}_{t}\big)-Y_{N}\big(X^{(N)}_{t}\big)   \Big|  \bigg] \  \leq \  CN^{-\frac{1}{\nu+1}}\mathbb{E}_{\widehat{x}_{N}}\bigg[\sup_{0\leq t\leq T}\big| X^{(N)}_{t} \big|^{\nu}  \bigg] \ =\ \mathcal{O}\big( N^{-\frac{1}{\nu+1}}  \big)\ .
\end{align*}
The order equality follows from part  (1).   

Since the initial values $\widehat{y}_{N}:=  Y_{N}(\widehat{x}_{N})  $ lie on a compact set for $N>1$, it is sufficient for us to show that for any $\epsilon, \delta>0$ we can pick $n>1$ large enough so that
\begin{align}\label{Tapity}
\limsup_{N\rightarrow \infty} \mathbb{P}_{\widehat{y}_{N}}\bigg[ \sup_{ 0\leq m <n  } \sup_{0\leq t\leq \frac{T}{n}}\Big|M^{(N)}_{t+\frac{mT}{n}}   -M^{(N)}_{\frac{mT}{n}}      \Big| >\delta \bigg] \ < \  \epsilon  \ .
\end{align}
The above condition for tightness follow easily, for instance, from~\cite[Theorem 7.3]{Billingsly}.   By Chebyshev's and Jensen's inequalities, we have the first inequality below:
\begin{align}
 \mathbb{P}_{\widehat{y}_{N}}\bigg[ \sup_{ 0\leq m <n  } \sup_{0\leq t\leq \frac{T}{n}}\Big|M^{(N)}_{t+\frac{mT}{n}}   -M^{(N)}_{\frac{mT}{n}}  \Big| >\delta \bigg] \ \leq \  &  \frac{1}{\delta}\mathbb{E}_{\widehat{y}_{N}}\Bigg[ \sum_{ m=0  }^{n-1} \sup_{0\leq t\leq \frac{T}{n}}\Big|M^{(N)}_{t+\frac{mT}{n}}   -M^{(N)}_{\frac{mT}{n}}   \Big|^{4} \Bigg]^{\frac{1}{4}}\nonumber  \\    = \  &   \frac{1}{\delta}\mathbb{E}_{\widehat{y}_{N}}\Bigg[ \sum_{ m=0  }^{n-1}\mathbb{E}\bigg[ \sup_{0\leq t\leq \frac{T}{n}}\Big|M^{(N)}_{t+\frac{mT}{n}}   -M^{(N)}_{\frac{mT}{n}}    \Big|^{4}\,\bigg|\,  \mathcal{F}_{\frac{mT}{n}}^{(N)}  \bigg]    \Bigg]^{\frac{1}{4}} \nonumber  \\  \  \leq  &   \frac{C}{\delta}\mathbb{E}_{\widehat{y}_{N}}\Bigg[ \sum_{ m=0  }^{n-1}\mathbb{E}\bigg[ \Big|M^{(N)}_{\frac{(m+1)T}{n}}   -M^{(N)}_{\frac{mT}{n}}   \Big|^{4}\,\bigg|\,  \mathcal{F}_{\frac{mT}{n}}^{(N)}  \bigg]    \Bigg]^{\frac{1}{4}}\ , \label{Tintin}
\end{align}
 where $\mathcal{F}_{r}^{(N)}$ is the information known about the process $(M^{(N)}_{t})_{t\geq 0}$ up to time $r\in \R^{+}$.  The second inequality above is an application of Doob's maximal inequality to each conditional expectation.  By Rosenthal's inequality, (\ref{Tintin}) is bounded by
\begin{align} 
   \leq \   &   \frac{C }{\delta }\mathbb{E}_{\widehat{y}_{N}}\Bigg[ \sum_{ m=0  }^{n-1}\mathbb{E}\bigg[ \bigg( \int_{\frac{mT}{n}}^{\frac{(m+1) T}{n}} dr Q_{N}\big(  M_{r}^{(N)} \big) \bigg)^2\,+\, \sup_{\frac{mT}{n}\leq  r \leq \frac{(m+1)T}{n} } \big|\Delta_{r}^{(N)}\big|^4  \,\bigg|\,  \mathcal{F}_{\frac{mT}{n}}^{(N)}  \bigg]          \Bigg]^{\frac{1}{4}}\, , \nonumber  
\intertext{where $\Delta_{r}^{(N)}:=  M_{r^+}^{(N)} -M_{r^-}^{(N)}      $.   Applying Lemma~\ref{LemJumpSize} to the second term, the above is smaller than    }
 \leq \   &   \frac{C }{\delta }\mathbb{E}_{\widehat{y}_{N}}\Bigg[ \sum_{ m=0  }^{n-1} \mathbb{E}\bigg[ \frac{T }{ n } \int_{\frac{mT}{n}}^{\frac{(m+1) T}{n}} dr\big|Q_{N}\big(  M_{r}^{(N)} \big) \big|^{2}\,+\, N^{-\frac{4}{1+\nu}} \sup_{\frac{mT}{n}\leq  r \leq \frac{(m+1)T}{n} } \big| M_{r}^{(N)}  \big|^{\frac{4\nu}{\nu +1 }}\,\bigg|\,  \mathcal{F}_{\frac{mT}{n}}^{(N)}  \bigg]          \Bigg]^{\frac{1}{4}}\,  \nonumber  \\
   = \  &   \frac{C }{\delta  }\mathbb{E}_{\widehat{y}_{N}}\bigg[ \frac{T }{ n }\sup_{0\leq r\leq T}\big|Q_{N}\big(  M_{r}^{(N)} \big) \big|^{2}\,+\, N^{-\frac{4}{1+\nu}} \sup_{0\leq  r \leq T} \big| M_{r}^{(N)}  \big|^{\frac{4\nu}{\nu +1 }}           \bigg]^{\frac{1}{4}} \nonumber \\
   \leq \  &   \frac{C }{\delta  }\mathbb{E}_{\widehat{y}_{N}}\bigg[ \frac{T }{ n }\sup_{0\leq r\leq T}\Big(1+ \big| M_{r}^{(N)}\big|^{\frac{\nu}{\nu+1}}  \Big)^{2}\,+\, N^{-\frac{4}{1+\nu}} \sup_{0\leq  r \leq T} \big| M_{r}^{(N)}  \big|^{\frac{4\nu}{\nu +1 }}           \bigg]^{\frac{1}{4}}
\label{BoDiddy} \ ,  \end{align}
where the last inequality uses part (4) of Lem.~\ref{SecLemMisc} to bound the first term. 

Finally, applying part (1) to~(\ref{BoDiddy}) we can obtain an inequality of the form 
 \begin{align}\nonumber
 \mathbb{P}_{\widehat{y}_{N}}\bigg[ \sup_{ 0\leq m <n  } \sup_{0\leq t\leq \frac{T}{n}}\Big|M^{(N)}_{t+\frac{mT}{n}}   -M^{(N)}_{\frac{mT}{n}}  \Big| >\delta \bigg] \  \leq  \  &  \frac{C}{\delta}\Big( n^{-\frac{1}{4}}\,+\, N^{-\frac{1}{\nu+1}}    \Big)        \ .
\end{align}
 Thus we can pick $n, N\gg 1$ to be large enough to make the left-hand side of (\ref{Tapity}) arbitrarily small.  \end{proof}

\end{document}